\documentclass{scrartcl}
\pdfoutput=1

\title{Configuration spaces of clusters\\as $E_d$-algebras}
\author{Florian Kranhold}
\KOMAoptions{headings=standardclasses,numbers=enddot}%

\usepackage[utf8]{inputenc}
\usepackage[T1]{fontenc}

\usepackage[british]{babel}
\usepackage{csquotes}                         %
\usepackage[UKenglish]{isodate}
\cleanlookdateon                              %

\usepackage{lmodern}

\usepackage{amssymb,amsmath}           %
\usepackage{mathtools,letltxmacro}     %
\usepackage{stmaryrd}                  %

\usepackage[osf,sc]{mathpazo}

\usepackage{myT1classico}

\usepackage{setspace}
\usepackage{ellipsis}

\usepackage[babel,tracking=true]{microtype}
\makeatletter
\@ifpackagelater{microtype}{2020/12/08}
  {\microtypesetup{nopatch=footnote}}  %
  {}                                   %
\makeatother

\UseMicrotypeSet[tracking]{smallcaps}
\SetTracking{encoding=*,shape=sc}{30}

\makeatletter
  \widowpenalty=\@M
\makeatother

\usepackage{scalerel}

\usepackage{relsize}                                   %

\usepackage[shortlabels]{enumitem}

\usepackage{graphicx}
\usepackage[margin    =20pt,
            font      ={footnotesize,stretch=1.06},
            format    =plain,
            labelsep  =period,
            labelfont =bf]{caption}
\usepackage{subcaption}

\usepackage{booktabs}

\usepackage{amsthm}

\newcommand{\tref}[2]{#1~\ref{#2}}

\deffootnote[1.8em]{1em}{1.8em}{%
   \textsuperscript{\thefootnotemark}}

\addtokomafont{pagehead}{\fontfamily{ppl}\selectfont}

\usepackage{etoolbox}
\ifundef{\abstract}{}{
  \patchcmd{\abstract}{\quotation}{\quotation\noindent\ignorespaces}{}{}}

\recalctypearea 

\usepackage{itemize}

\usepackage[pdfusetitle,%
            bookmarksnumbered=true,%
            unicode,%
            hidelinks]{hyperref}
\newcommand{\mail}[1]{\upshape\href{mailto:#1}{\textsf{#1}}}

 \usepackage{xcolor}
 \hypersetup{
   colorlinks,
   linkcolor={black!45!red},
   citecolor={black!55!green},
   urlcolor ={black!45!blue}
 }

\usepackage[%
            backend      = biber,
            style        = alphabetic,
            maxbibnames  = 9,
            maxcitenames = 9,
            bibencoding  = utf8,
            datamodel    = ext-eprint]{biblatex}

\DefineBibliographyExtras{british}{%
  \renewcommand{\finalandcomma}{,}
}

\DefineBibliographyStrings{english}{
  backrefpage={↑},
  backrefpages={↑}
}

\DeclareSourcemap{
  \maps[datatype=bibtex]{
    \map{
      \step[fieldsource=pmid, fieldtarget=pubmed]
    }
  }
}

\DeclareFieldFormat{urn}{%
  \textsc{urn}\addcolon\space
  \ifhyperref
  {\href{https://nbn-resolving.org/#1}{\nolinkurl{#1}}}
  {\nolinkurl{#1}}}

\DeclareFieldFormat{sdr}{%
  \textsc{sdr}\addcolon\space
  \ifhyperref
  {\href{https://purl.stanford.edu/#1}{\nolinkurl{#1}}}
  {\nolinkurl{#1}}}

\DeclareFieldFormat{zbl}{%
  \textsc{zbl}\addcolon\space
  \ifhyperref
  {\href{https://zbmath.org/#1}{\nolinkurl{#1}}}
  {\nolinkurl{#1}}}

\DeclareFieldFormat{hdl}{%
  \textsc{hdl}\addcolon\space
  \ifhyperref
  {\href{https://hdl.handle.net/#1}{\nolinkurl{#1}}}
  {\nolinkurl{#1}}} 

\renewbibmacro*{eprint}{%
  \printfield{zbl}%
  \newunit\newblock
  \printfield{urn}%
  \newunit\newblock
  \printfield{sdr}%
  \newunit\newblock
  \printfield{hdl}%
  \newunit\newblock
  \iffieldundef{eprinttype}
  {\printfield{eprint}}
  {\printfield[eprint:\strfield{eprinttype}]{eprint}}}

\DeclareFieldFormat[misc]{title}{‘#1’}                  %
\AtBeginBibliography{\small}

\bibliography{bibliography.bib}

\usepackage{xpatch}
\xpatchcmd\swappedhead{~}{.~}{}{}

\theoremstyle{plain}
\newtheorem{theo}   {Theorem}[section]
\newtheorem{lem}    [theo]{Lemma}

\newtheorem{cor}    [theo]{Corollary}

\theoremstyle{definition}
\newtheorem{defi}   [theo]{Definition}

\newtheorem{rem}    [theo]{Remark}
\newtheorem{expl}   [theo]{Example}
\newtheorem{constr} [theo]{Construction}
\newtheorem{nota}   [theo]{Notation}

\newtheoremstyle{quote}%
{-18pt}%
{0pt}%
{\itshape}%
{0pt}%
{\bfseries}%
{}%
{5px}%
{\thmname{#1}\thmnumber{\hspace*{3px}#2}\thmnote{ (#3)}.}
\theoremstyle{quote}
\newtheorem{qThm}{Theorem}

\newtheoremstyle{Alph}%
{0pt}%
{.5\baselineskip}%
{\itshape}%
{0pt}%
{\bfseries}%
{}%
{5px}%
{\thmname{#1}\thmnumber{\hspace*{3px}#2}\thmnote{ (#3)}.}
\theoremstyle{Alph}
\newtheorem{aThm}{Theorem}
\newlength{\fixmidfigure}
\newenvironment{tFigure}{
  \setlength{\fixmidfigure}{\lastskip}\addvspace{-\lastskip}
  \begin{figure}[t]
  }{
  \end{figure}
  \addvspace{\fixmidfigure}
}

\usepackage{tikz-cd}
\usepackage{tikz}
\usetikzlibrary{
  arrows,
  calc
}
\tikzset{
  }

\let \le       \leqslant
\let \ge       \geqslant

\let \setminus \smallsetminus
\let \emptyset \varnothing
\let \phi      \varphi
\let \epsilon  \varepsilon
\let \theta    \vartheta
\let \binom    \tbinom

\LetLtxMacro\orgvdots\vdots
\LetLtxMacro\orgddots\ddots
\makeatletter
\DeclareRobustCommand\vdots{%
  \mathpalette\@vdots{}%
}
\newcommand*{\@vdots}[2]{%
  \sbox0{$#1\cdotp\cdotp\cdotp\m@th$}%
  \sbox2{$#1.\m@th$}%
  \vbox{%
    \dimen@=\wd0 %
    \advance\dimen@ -3\ht2 %
    \kern.5\dimen@
    \dimen@=\wd2 %
    \advance\dimen@ -\ht2 %
    \dimen2=\wd0 %
    \advance\dimen2 -\dimen@
    \vbox to \dimen2{%
      \offinterlineskip
      \copy2 \vfill\copy2 \vfill\copy2 %
    }%
  }%
}
\makeatother

\newcommand{\bm}  [1]{\mathbold{#1}}
\newcommand{\quot}[2]{\left.\raise0.5ex\hbox{$#1$} \right/\hspace*{-2px}\lower0.5ex\hbox{$#2$}}

\newcommand{\on}  [1]{\mathrm{#1}}

\usepackage[outline]{contour}
\usepackage[normalem]{ulem}

\contourlength{1pt}
\newcommand{\ul}[1]{%
  \uline{\smash{\phantom{#1}}}%
  \llap{\contour{white}{$#1$}}%
}
\newcommand{\ulp}[1]{\uline{\smash{#1}}}

\newcommand{\pa}     [1]{\mathopen{}\left(#1\right)\mathclose{}}                    %
\newcommand{\abs}    [1]{\mathopen{}\left\lvert #1\right\rvert\mathclose{}}         %
\newcommand{\set}    [1]{\mathopen{}\left\{#1\right\}\mathclose{}}                  %
\newcommand{\pr}  {\on{pr}}

\newcommand{\op}  {{\on{op}}}

\newcommand{\Inj} {\mathbf{Inj}}
\newcommand{\Top} {\mathbf{Top}}

\makeatletter
\newcommand\dirlim{\mathop{\mathpalette\varlim@{\rightarrowfill@\scriptscriptstyle}}\nmlimits@}
\newcommand\invlim{\mathop{\mathpalette\varlim@{\leftarrowfill@\scriptscriptstyle}}\nmlimits@}
\makeatother

\newcommand{\map}{\on{map}}

\newcommand{\swr}{\kern.6px\wr\kern.6px}

\DeclareSymbolFont{eulargesymbols}{U}{zeuex}{m}{n}
\DeclareMathSymbol{\intop}{\mathop}{eulargesymbols}{"52}

\def\centerarc[#1](#2)(#3:#4:#5)%
    { \draw[#1] ($(#2)+({#5*cos(#3)},{#5*sin(#3)})$) arc (#3:#4:#5); }

\usepackage[scr=boondoxo, scrscaled=1]{mathalfa}       %

\renewcommand{\a}   {\mathfrak{a}}
\renewcommand{\b}   {\mathfrak{b}}
\ifdefined\C\renewcommand{\C}{\mathbb{C}}\else\newcommand{\C}{\mathbb{C}}\fi
\newcommand  {\E}   {\mathcal{E}}
\newcommand  {\e}   {\mathbf{e}}
\renewcommand{\i}   {\mathbf{i}}
\newcommand  {\K}   {\mathbb{K}}
\renewcommand{\L}   {\mathbb{L}}
\newcommand  {\N}   {\mathbb{N}}
\newcommand  {\n}   {\mathfrak{n}}
\renewcommand{\P}   {\mathcal{P}}
\newcommand  {\p}   {\mathfrak{p}}
\newcommand  {\Q}   {\mathbb{Q}}
\newcommand  {\q}   {\mathfrak{q}}
\newcommand  {\R}   {\mathbb{R}}
\renewcommand{\S}   {\mathbb{S}}
\newcommand  {\V}   {\mathscr{V}}
\newcommand  {\X}   {\mathbb{X}}
\newcommand  {\Z}   {\mathbb{Z}}

\newcommand{\bbD}{\mathbb{D}}
\newcommand{\bbE}{\mathbb{E}}

\newcommand{\bE}{\mathbb{E}}

\newcommand{\bS}{\mathbb{S}}

\newcommand{\scC}{\mathscr{C}}

\newcommand{\scF}{\mathscr{F}}

\newcommand{\scO}{\mathscr{O}}

\newcommand{\bfC}{\mathbf{C}}

\newcommand{\bmX}{\bm{X}}

\newcommand{\frG}{\mathfrak{G}}

\newcommand{\frS}{\mathfrak{S}}

\ifdefined\frq\renewcommand{\frq}{\mathfrak{q}}\else \newcommand{\frq}{\mathfrak{q}}\fi

\newcommand{\tC}{\tilde{C}}

\DeclareMathOperator{\supp}{supp}

\definecolor{ddgreen}{rgb}{.20,.60,.20}

\definecolor{dgrey}  {rgb}{.45,.45,.45}
\definecolor{dyellow}{rgb}{.84,.69,.00}
\definecolor{dred}   {rgb}{.90,.30,.30}
\definecolor{dblue}  {rgb}{.28,.32,.80}
\definecolor{dgreen} {rgb}{.17,.65,.16}

\definecolor{bgrey}  {rgb}{.85,.85,.85}
\definecolor{byellow}{rgb}{.95,.93,.40}
\definecolor{bred}   {rgb}{.95,.60,.60}
\definecolor{bblue}  {rgb}{.50,.50,.90}
\definecolor{bgreen} {rgb}{.70,.90,.60}
\definecolor{blila}  {rgb}{.96,.84,.96}

\definecolor{bbyellow}{rgb}{.95,.92,.60}
\definecolor{bbred}   {rgb}{.98,.85,.85}
\definecolor{bbblue}  {rgb}{.70,.70,.96}
\definecolor{bbgreen} {rgb}{.80,.96,.70}
\definecolor{bblila}  {rgb}{.94,.84,.97}

\definecolor{dlila}  {rgb}{.70,.30,.70}
\definecolor{lila}   {rgb}{.80,.60,.90}
\definecolor{turq}   {rgb}{.3,.7,.7}
\definecolor{grey}   {rgb}{.65,.65,.65}
\definecolor{dorange}   {rgb}{.9,.4,.2}

\numberwithin{equation}{section}

\usepackage{scalerel}

\usepackage{scrlayer-scrpage}
\lohead{Configuration spaces of clusters as $E_d$-algebras}
\cohead{}
\rohead{F.\ Kranhold}
\begin{document}

\maketitle

\begin{abstract}
  It is a classical result that configuration spaces of labelled particles in
  $\R^d$ are free $E_d$-algebras and that their $d$-fold bar construction is
  equivalent to the $d$-fold suspension of the labelling space.
  
  In this paper, we study a variation of these spaces, namely configuration
  spaces of labelled \emph{clusters} of particles. These
  configuration spaces are again $E_d$-algebras, and we give geometric models
  for their iterated bar construction in two different ways: one establishes a
  description of these configuration spaces of clusters as \emph{cellular}
  $E_1$-algebras, and the other one uses an additional \emph{verticality}
  constraint. In the last section, we apply these results in order to
  calculate the stable homology of certain \emph{vertical} configuration spaces.
\end{abstract}

\section{Introduction and overview}
\label{sec:introduction}
Let us start with the classical definition of configuration spaces: for a space
$E$ and a natural number $r\ge 0$, the \emph{ordered configuration space of $r$
  particles in $E$} is defined to be
\[\tilde{C}_r(E)\coloneqq \set{\pa{z_1,\dotsc,z_r}\in E^r;\,z_i\ne z_j\text{ for }i\ne j}.\]
The $r$\textsuperscript{th} symmetric group $\frS_r$ acts freely on $\tC_r(E)$
by permuting coordinates, and we call the quotient
$C_r(E)\coloneqq \tC_r(E)/\frS_r$ the \emph{unordered configuration space of $r$
  particles in $E$}.

For a based space $X$, we define the \emph{labelled configuration space}
$C(E;X)$ as the union of all $\tC_r(E)\times_{\frS_r} X^r$, quotiented by the
relation that identifies $[z_1,\dotsc,z_r;x_1,\dotsc,x_r]$ with
$[z_1,\dotsc,\widehat{z_i},\dotsc,z_r;x_1,\dotsc,\widehat{x_i},\dotsc,x_r]$ if
$x_i$ is the basepoint of $X$. Visually, each particle $z_i$
carries a label $x_i\in X$, and if the label reaches the basepoint, then this
particle vanishes.

For the case $E=\R^d$ with $d\ge 1$, the labelled configuration
space $C(\R^d;X)$ is an $E_d$-algebra, more precisely: it admits an action of
the little $d$-cubes operad $\scC_d$ by inserting configurations into boxes
\cite{May-1972}. It is even equivalent to the \emph{free} $E_d$-algebra over
$X$, and its $d$-fold bar construction is equivalent to the $d$-fold
suspension $\Sigma^d X$, see \cite{Segal-1973}.

\noindent This paper studies variations of these labelled configuration spaces,
which additionally carry the information that some of the particles ‘belong
together’—that is: they form a \emph{cluster}—, and investigates their structure
as $E_d$-algebras.

\begin{quote}
  \textbf{Definition.} Let $E$ be a space and let $k\ge 1$ be an integer. A
  \emph{$k$-cluster} is an ordered configuration of $k$ distinct particles in
  $E$. For each integer $r\ge 0$, let $\tC_{\smash{r\times k}}(E)$ be the space
  of $r$ disjoint $k$-clusters. Then $\frS_k\wr\frS_r$ acts on
  $\tC_{\smash{r\times k}}(E)$ by permuting particles within the same cluster,
  and by permuting clusters. The quotient $C_{\smash{r\times k}}(E)$ is called
  the \emph{configuration space of $r$ $k$-clusters in $E$}.\looseness-1
\end{quote}

\noindent Intuitively, $C_{\smash{r\times k}}(E)$ parametrises unordered
collections of pairwise disjoint subsets of $E$, all of cardinality $k$, see the
first case of \tref{Figure}{fig:manyV}: it is a covering space of
$C_{\smash{r\cdot k}}(E)$.  The above definition has a labelled
counterpart (see \tref{Definition}{def:labClustConf} for details):\looseness-1

\begin{quote}
  \textbf{Definition.} For a well-based space $X$, we define the
  configuration space $C^k(E;X)$ of unordered $k$-clusters in $E$; each cluster
  carries a label inside $X$, and if the label reaches the basepoint, then the
  entire cluster vanishes.\looseness-1
\end{quote}

\begin{figure}[t]
  \centering
  \begin{tikzpicture}[scale=2.4]
  \draw[very thin, dblue!40] (.3,0) -- (.3,1);
  \draw[very thin, dblue!40] (.7,0) -- (.7,1);
  \draw (-1.5,0) rectangle (-.5,1);
  \node at (-1,-.16) {\scriptsize $C_{3\times 2}(\R^{2})$};
  \draw (0,0) rectangle (1,1);
  \draw[black!35,thick] (.3,.2) -- (.25,.2) -- (.25,.7) -- (.3,.7);
  \draw[black!35,thick] (.3,.4) -- (.35,.4) -- (.35,.8) -- (.3,.8);
  \draw[black!35,thick] (.7,.3) -- (.7,.6);
  \fill (.3,.2) circle (.02);
  \fill (.3,.4) circle (.02);
  \fill (.3,.7) circle (.02);
  \fill (.3,.8) circle (.02);
  \fill (.7,.3) circle (.02);
  \fill (.7,.6) circle (.02);
  \node at (.5,-.16) {\scriptsize $C_{3\times 2}(\R^{1,1})$};
  \draw[black!35,thick] (-1.3,.4) -- (-1.2,.2);
  \draw[black!35,thick] (-1.3,.7) -- (-.8,.6);
  \draw[black!35,thick] (-1,.8) -- (-.7,.3);
  \fill (-1.2,.2) circle (.02);
  \fill (-1.3,.4) circle (.02);
  \fill (-1.3,.7) circle (.02);
  \fill (-1,.8) circle (.02);
  \fill (-.7,.3) circle (.02);
  \fill (-.8,.6) circle (.02);
  \draw[black!30] (1.75,1) -- (1.75,.28) -- (2.7,.28);
  \draw (1.75,1) -- (2.7,1) -- (2.7,.28);
  \draw[thick, white] (1.9,.1) -- (1.9,.842);
  \draw[thick, white] (2.3,.2) -- (2.3,.938);
  \draw[very thin, dblue!40] (1.9,.1) -- (1.9,.842);
  \draw[very thin, dblue!40] (2.3,.2) -- (2.3,.938);
  \draw[black!35,thick] (1.9,.4) -- (1.9,.6);
  \draw[black!35,thick] (2.3,.5) -- (2.34,.5) -- (2.34,.85) -- (2.3,.85);
  \draw[black!35,thick] (2.3,.4) -- (2.26,.4) -- (2.26,.65) -- (2.3,.65);
  \fill (1.9,.4) circle (.02);
  \fill (1.9,.6) circle (.02);
  \fill (2.3,.5) circle (.02);
  \fill (2.3,.85) circle (.02);
  \fill (2.3,.4) circle (.02);
  \fill (2.3,.65) circle (.02);
  \draw[white,line width=.4mm] (1.5,0) rectangle (2.5,.75);
  \draw (1.5,0) rectangle (2.5,.75);
  \draw[black!30] (1.5,0) -- (1.75,.28);
  \draw (2.5,0) -- (2.7,.28);
  \draw (2.5,.75) -- (2.7,1);
  \draw (1.5,.75) -- (1.75,1);
  \node at (2,-.16) {\scriptsize $C_{3\times 2}(\R^{2,1})$};
  \draw[shift={(1.5,0)},black!30] (1.75,1) -- (1.75,.28) -- (2.7,.28);
  \draw[shift={(1.5,0)},black!30] (1.5,0) -- (1.75,.28);
  \draw[shift={(1.5,0)},white,line width=.4mm] (1.5,0) rectangle (2.5,.75);
  \draw[very thin, dblue!60,fill=dblue!15,shift={(1.5,0)},opacity=.7] (1.8,0) --
  (2.035,.28) -- (2.035,1) -- (1.8,.75) -- (1.8,0);
  \draw[very thin, dblue!60,fill=dblue!15,shift={(1.5,0)},opacity=.7] (2.2,0) --
  (2.415,.28) -- (2.415,1) -- (2.2,.75) -- (2.2,0);
  \draw[shift={(1.5,0)},black!35,thick] (1.85,.2) -- (1.96,.4);
  \draw[shift={(1.5,0)},black!35,thick] (1.9,.6) -- (1.98,.85);
  \draw[shift={(1.5,0)},black!35,thick] (2.25,.4) -- (2.3,.6);
  \draw[shift={(1.5,0)}] (1.75,1) -- (2.7,1) -- (2.7,.28);
  \draw[shift={(1.5,0)},line width=.4mm,dblue!10.5] (1.85,.75) -- (2,.75);
  \draw[shift={(1.5,0)}] (1.5,0) rectangle (2.5,.75);
  \draw[shift={(1.5,0)}] (2.5,0) -- (2.7,.28);
  \draw[shift={(1.5,0)}] (2.5,.75) -- (2.7,1);
  \draw[shift={(1.5,0)}] (1.5,.75) -- (1.75,1);
  \fill[shift={(1.5,0)}] (1.85,.2) circle (.02);
  \fill[shift={(1.5,0)}] (1.96,.4) circle (.02);
  \fill[shift={(1.5,0)}] (1.98,.85) circle (.02);
  \fill[shift={(1.5,0)}] (1.9,.6) circle (.02);
  \fill[shift={(1.5,0)}] (2.25,.4) circle (.02);
  \fill[shift={(1.5,0)}] (2.3,.6) circle (.02);
  \node at (3.5,-.16) {\scriptsize $C_{3\times 2}(\R^{1,2})$};
\end{tikzpicture}
  \caption{Several configuration spaces of three $2$-clusters inside $\R^2$ and
    $\R^3$.}\label{fig:manyV}
\end{figure}

\noindent Both previous definitions can be given in slightly higher
generality, by allowing configuration of clusters with different sizes
and balancing the internal ordering of a cluster with a given symmetric
action on the labelling space: this is done in §\,\ref{sec:preliminaries}.

As one can easily see, the configuration space $C^k(\R^d;X)$ again admits the
structure of an $E_d$-algebra by inserting configurations of clusters into
boxes. One of our goals is to give a geometric interpretation of the $d$-fold
bar construction of $C^k(\R^d;X)$. While this seems to be hard in general, we
can give an answer in the case $d=1$: In §\,\ref{sec:cell}, we decompose
$C^k(\R;X)$ into ‘free components’, i.e.\ we give an $E_1$-cellular
decomposition in the sense of
\cite{Galatius-Kupers-RW-2018,Galatius-Kupers-RW-2019,Klang-Kupers-Miller}.  For
this purpose, we define what it means for a collection of clusters to be
\emph{entangled}. This gives rise to an $E_1$-filtration
$\scF_\bullet C^k(\R;X)$ such that each $\scF_w C^k(\R;X)$ arises from
$\scF_{w-1}C^k(\R;X)$ by attaching free $E_1$-algebras. Using that the bar
construction turns $E_1$-cell attachments into usual cell attachments, we
show:\looseness-1

\begin{quote}
  \begin{qThm}\label{thm:a}There is a weak
  equivalence $\smash{BC^k(\R;X)\simeq \Sigma \bigvee_{e} X^{\wedge \# e}}$,
  where $e$ ranges in a set of ‘entanglement types’
  (\tref{Definition}{def:entType}) and has a weight $\# e\ge 1$.
  \end{qThm}
\end{quote}
This result is surprising for two reasons: first, it shows that the attaching
maps for vanishing clusters simplify drastically when applying the bar
construction, and second, it implies that if $X$ path-connected, then
$C^k(\R;X)$ is equivalent to a free $E_1$-algebra.\looseness-1

\enlargethispage{\baselineskip}

In addition to that, we can give a partial answer for the cases $d\ge 2$. To
this aim, we introduce a slightly more general family of configuration spaces of
clusters (see also \tref{Definition}{def:vert}):

\begin{quote}
  \textbf{Definition.} For $0\le p\le d$, and $k$ and $r$ as above, we define
  the subspace
  $C_{\smash{r\times k}}(\R^{p\kern-.6px,\kern.6pxd-p})\subseteq
  C_{\smash{r\times k}}(\R^{d})$, called \emph{vertical configuration space},
  where particles within the same cluster share their first $p$ coordinates. For
  a based space $X$, we define the subspace
  \mbox{$C^k(\R^{p\kern-.6px,\kern.6pxd-p};X)\subseteq C^k(\R^{d};X)$} with the
  same constraint.
\end{quote}
\noindent Several of these configuration spaces are depicted in
\tref{Figure}{fig:manyV}. In the case $p=d-1$, we require that all particles of
the same cluster lie on a common vertical line; hence the terminology.  It is
not hard to see that
$C^k(\R^{p\kern-.4px,\kern.4pxd-p};X)\subseteq C^k(\R^{d};X)$ is an
$E_{d}$-subalgebra, and it turns out that the first $p$ delooping steps are
manageable by a straightforward adaption of the methods from Segal’s argument
\cite{Segal-1973}:

\begin{quote}
  \begin{qThm}\label{thm:b}
    $B^pC^k(\R^{p\kern-.4px,\kern.4pxd-p};X)\simeq C^k(\R^{d-p};\Sigma^p X)$
    as $E_{d-p}$-algebras.
  \end{qThm}
\end{quote}
Informally, this means that the first $p$ delooping steps ‘resolve’ the
verticality constraint. This is perhaps not so surprising: in the
$E_{d}$-algebra $C^k(\R^{p\kern-.4px,\kern.4pxd-p};X)$, clusters play the rôle
of particles in the classical labelled configuration space, and from the
perspective of the first $p$ coordinates, they also behave as such.  Theorems
\ref{thm:a} and \ref{thm:b} are special cases of Theorems \ref{thm:BCRX} and
\ref{thm:segal}, respectively: those also cover the case of configuration spaces
of clusters with different sizes and balanced labels.

Combining Theorems \ref{thm:a} and \ref{thm:b}, we obtain a model for the
iterated bar construction of the $E_{p+1}$-algebra
$C^k(\R^{p\kern-.3px,\kern.2px1};X)$. This can be used to calculate the
\emph{stable} homology of these spaces: it is shown in \cite{Bianchi-Kranhold}
that adding a new cluster
\mbox{$C_{\smash{r\times k}}(\R^{p\kern-.3px,\kern.2px1})\to
  C_{\smash{(r+1)\times k}}(\R^{p\kern-.3px,\kern.2px1})$} is homologically
stable for \mbox{$p\ge 1$}. We determine the stable homology
$H_\bullet(C_{\smash{\infty\times k}}(\R^{p\kern-.3px,\kern.2px1}))$ as follows:
there is a distinguished entanglement type $e_0$ (see
\tref{Definition}{def:entType}) corresponding to a \emph{single} $k$-cluster. To
each finitely supported family $\lambda=(\lambda_e)_{e\ne e_0}$ of integers
$\lambda_e\ge 0$, where $e$ ranges in the set of all entanglement types of
$k$-clusters, we assign a shifting parameter $s(\lambda)$ and a graded module
$M_\bullet(\R^{p+1};\lambda[\infty])$, which is the (twisted) stable homology of
a sequence of certain \emph{coloured} configuration spaces \cite{Palmer-2018},
the stabilisation step given by adding particles of colour $e_0$. We then show
the following:\looseness-1

\begin{quote}
  \textbf{Theorem~\ref{thm:stableH}.}\hspace*{5px}\itshape
  For each $p,k\ge 1$, we have an
  isomorphism of graded modules\vspace*{-3px} 
  \[\textstyle H_\bullet(C_{\smash{\infty\times k}}(\R^{p\kern-.3px,\kern.2px1}))\cong
    \bigoplus_{\lambda}
    M_{\bullet-p\cdot s(\lambda)}(\R^{p+1};\lambda[\infty]).\]
\end{quote}
\noindent The corresponding \emph{unstable} modules
$M_\bullet(\R^{p+1};\lambda[n])$ already appeared in \cite{Bianchi-Kranhold} to
describe the homology of certain filtration quotients of
$C_{\smash{r\times k}}(\R^{\smash{p\kern-.3px,\kern.2px1}})$; however, it
remained open if the associated spectral sequence collapses on its first page
and if the extension problem is trivial. Theorem~\ref{thm:stableH} tells us that
this is at least stably the case.\looseness-1

\paragraph{Related work}
Configuration spaces of clusters have been studied from the perspective of
homological stability \cite{Tran,Palmer-2021} and in relation to Hurwitz
spaces \cite{Tietz}. Moreover, the ‘clustering’ of particles is useful to
describe an enhancement of the little $d$-cubes operad $\scC_d$ that acts on
moduli spaces of manifolds with \emph{multiple} boundary components; this is a
leading principle in the author’s PhD thesis \cite{Kranhold-2022}.

Vertical configuration spaces, especially their higher homotopy groups and their
homological stability, have been studied in
\cite{Herberz,Roesner,Latifi,Bianchi-Kranhold}. They are also closely related to
\emph{fibrewise} configuration spaces, which appear in \cite{Cnossen} in order
to formulate an approximation theorem for configurations with twisted labels and
labels in partial abelian monoids. Moreover, these spaces assemble into a
coloured operad $\V_{p\kern-.6px,\kern.6pxd-p}$, which is similar to the
extended Swiss cheese operad \cite{Willwacher} and acts on moduli spaces of
$d$-dimensional manifolds with $p$-dimensional foliations, see
\cite{CFB-1990b,Kranhold-2022} for the case of surfaces with a $1$-dimensional
foliation.

\paragraph{Outlook}%
We still do not know what the iterated bar construction of the full
$E_d$-algebra $C^k(\R^d;X)$ looks like for $d\ge 2$. One might try to enhance
the $E_d$-cellular methods for the case $d=1$ to the general case; however, we
lack a good notion of higher-dimensional entanglement types. On the other hand,
it would already be interesting to know if the $E_d$-algebra $C^k(\R^d;X)$ is
equivalent to a free $E_d$-algebra if $X$ is path-connected.\looseness-1

\paragraph{Acknowledgements}%
I would like to thank my PhD supervisor, Carl-Friedrich Bödig\-heimer, for
suggesting the study of vertical configuration spaces and his valuable
advice. The article benefited from discussions with Andrea Bianchi about various
delooping methods, from discussions with Bastiaan Cnossen about several indexing
categories. Finally, I thank the anonymous referee for many useful
comments.\looseness-1

\paragraph{Funding}%
The author was supported by the \emph{Max Planck Institute for Mathematics} in
Bonn, by the \emph{Deutsche Forschungsgemeinschaft} (\textsc{dfg}, German
Research Foundation) under Germany’s Excellence Strategy (\textsc{exc}-2047/1,
390685813) and by the \emph{Promotionsförderung} of the
\emph{Studienstiftung des Deutschen Volkes}.

\section{Basic constructions}
\label{sec:preliminaries}
In this section, we formally introduce the aforementioned configuration spaces
of (labelled) clusters and make their $E_d$-algebra structure explicit.  As
already mentioned in the introduction, this can be done without much further
effort in slightly higher generality, by allowing configurations of clusters
with different sizes at the same time.

\begin{defi}\label{def:clustConf}
  Let $E$ be a space and $K=(k_1,\dotsc,k_r)$ be a tuple of integers $k_i\ge 1$.
  We start with a reindexing and let $\tC_K(E)\coloneqq \tC_{k_1+\dotsb+k_r}(E)$,
  but we denote its elements by tuples $(\vec z_1,\dotsc,\vec z_r)$,
  where $\vec z_i=(z_{i,1},\dotsc,z_{i,k_i})$, and we call $\vec z_i$ a
  \emph{cluster of size $k_i$}.

  If we denote by $r(k)\ge 0$ the number of occurrences of $k$ in the tuple $K$,
  then the group $\frS_K \coloneqq \prod_{k\ge 1}\frS_k\wr\frS_{r(k)}$ acts on
  $\tC_K(E)$ by permuting clusters of the same size and by permuting the
  internal ordering of each cluster. We call the quotient
  $C_K(E)\coloneqq \tC_K(E)/\frS_K$ the \emph{configuration space of clusters in
    $E$} and denote elements in $C_K(E)$ as (unordered) sums
  $\sum_{i=1}^r [\vec z_i]$ of unordered clusters
  $[\vec z_i]= [z_{i,1},\dotsc,z_{i,k_i}]
  =\{z_{i,1},\dotsc,z_{i,k_i}\}$.\looseness-1
\end{defi}

\begin{expl}
  Let $k\ge 1$ and $r\ge 0$ be integers. If $r\times k\coloneqq (k,\dotsc,k)$
  denotes the tuple of length $r$, then $C_{\smash{r\times k}}(E)$ is exactly the
  configuration space from the introduction.
\end{expl}

Labelled configuration spaces of clusters should generalise the classical notion
of a labelled configuration space in the following way: we want to assign to
each $k$-cluster a label inside a based space $X_k$ and balance the internal
ordering of each cluster with a given symmetric action on $X_k$.  In order to
make this definition precise, we first have to introduce an indexing category,
which is a special case of the Grothendieck construction and generalises the
notion of a wreath product $G\wr\frS_r=G^r\rtimes \frS_r$.
  
\begin{defi}
  Let $\Inj$ be the small category with objects $\ul{r}\coloneqq \{1,\dotsc,r\}$
  for all non-negative integers $r\in\{0,1,2,\dotsc\}$, and with morphisms
  $\ul{r}\to \ul{r}'$ being all injective maps of finite sets.  Then $\Inj$ is
  spanned by two sorts of maps: on the one hand \emph{permutations}
  $\tau\in\frS_r$, and on the other hand, the \emph{top cofaces}
  $d^r\colon \ul{r-1}\to \ul{r}$, where for each $1\le i\le r$, we denote by
  $d^i$ the unique strictly monotone function whose image does not contain the
  element $i\in \ul{r}$. Whenever we apply a contravariant functor to $\Inj$, we
  write $d_i\coloneqq (d^i)^*$.
\end{defi}

\begin{nota}[Tuples]\label{nota:tuples}
  Let $K=(k_1,\dotsc,k_r)$ be a tuple of integers $k_i\ge 1$.
  \begin{enumerate}[joinedup,packed]
  \item We denote by
    $|K|\coloneqq k_1+\dotsb+k_r$ the \emph{size} and by $\# K\coloneqq r$ the
    \emph{length} of $K$.
  \item For a sequence $\bmX\coloneqq (X_k)_{k\ge 1}$ in a complete
    category, we let $\bmX^K\coloneqq X_{k_1}\times\dotsb\times X_{k_r}$.
  \item For any injection $u\colon \ul{t}\hookrightarrow \ul{r}$ we define the \emph{pullback}
    $u^*K\coloneqq (k_{\smash{u(1)}},\dotsc,k_{\smash{u(t)}})$.
  \item $\frS_r$ acts on the set of $r$-tuples by pullback and we denote the orbit of $K$ by $[K]$.
  \end{enumerate}
\end{nota}

\begin{defi}[Wreath products]
  Let $\frG=(\frG_k)_{k\ge 1}$ be a sequence of discrete groups. We define
  the \emph{wreath product} $\frG\wr\Inj$ as the following small category:
  \begin{enumerate}[joinedup,packed]
  \item objects are tuples $K=(k_1,\dotsc,k_r)$ with
    $r\ge 0$ and $k_i\ge 1$ an integer;
  \item morphisms $K\to L$ are pairs $(u,\bm{g})$, where $\bm{g}\in \frG^K$
    and $u\colon \ul{r}\hookrightarrow \ul{s}$ with $K=u^*L$.
  \item composition is given by
    $(v,\bm{h})\circ (u,\bm{g}) \coloneqq (v\circ u,u^*\bm{h}\cdot \bm{g})$.
  \end{enumerate} 
\end{defi}

\begin{constr}
  Let $\frG=(\frG_k)_{k\ge 1}$ be a sequence of discrete groups and
  $\bmX=(X_k)_{k\ge 1}$ be a \emph{based $\frG$-sequence}, i.e.\ a sequence of
  based spaces, together with basepoint-preserving actions of $\frG_k$ on
  $X_k$. Then we obtain a functor to the category of topological spaces
  \[\bmX^-\colon \frG\wr\Inj\longrightarrow \Top,\quad K\mapsto \bmX^K=X_{k_1}\times\dotsb\times X_{k_r}\]
  as follows: for each injective map $u\colon \ul{r}\hookrightarrow \ul{s}$,
  each fibre has at most one element and we put
  $u_*(x_1,\dotsc,x_r)\coloneqq (x_{\smash{u^{-1}(1)}},\dotsc,x_{\smash{u^{-1}(s)}})$, where we
  define $x_\emptyset$ to be the basepoint. Moreover, $\frG^K$ acts on $\bmX^K$
  component-wise and we put
  $(u,\bm{g})_*(\bm{x})\coloneqq u_*(\bm{g}\cdot \bm{x})$ for $\bm{x}\in\bmX^K$.
\end{constr}

\begin{defi}\label{def:labClustConf}
  Consider the sequence $\frS\coloneqq (\frS_k)_{k\ge 1}$ of symmetric groups.
  Then, for each space $E$, the family of spaces $\tC_K(E)$ constitutes a functor
  $(\frS\wr\Inj)^\op\longrightarrow \Top$ by permuting clusters of the same size, by
  permuting the internal ordering of each cluster, and by declaring that for
  each $1\le i\le r$, the face map $d_i\colon \tC_K(E)\to \tC_{d_iK}(E)$ forgets
  the $i$\textsuperscript{th} cluster. If we are additionally given a based
  symmetric sequence $\bmX=(X_k)_{k\ge 1}$, then we define the
  \emph{configuration space of labelled clusters} $C(E;\bmX)$ to be the coend
  \begin{align*}
    C(E;\bm{X}) &\coloneqq \int^{K\in \frS\swr\Inj} \tC_K(E)\times \bm{X}^K
    \\ &= \on{coeq}
    \mathopen{}\bigg(\hspace*{-3px}
    \begin{tikzcd}[column sep=2.5em,ampersand replacement=\&]
      \displaystyle\coprod_{K,L}\tC_{L}(E)\times (\frS\wr\Inj)\binom{K}{L}\times \bmX^{K}
      \ar[shift left=1,-to]{r}{\alpha}\ar[shift right=1,-to]{r}[swap]{\beta} \&
      \displaystyle\coprod_K \tC_K(E)\times \bmX^{K}
    \end{tikzcd}\hspace*{-4px}\bigg)\mathclose{},
  \end{align*}
  where $(\frS\wr\Inj)\binom{K}{L}$ is the set of morphisms $w\colon K\to L$ and where
  $\alpha(\bm{z},w,\bm{x})=(w^*\bm{z},\bm{x})$ and $\beta(\bm{z},w,\bm{x})=(\bm{z},w_*\bm{x})$.
  Each tuple $(\vec z_1,\dotsc,\vec z_r,x_1,\dotsc,x_r)$ in $\tC_K(E)\times \bm{X}^K$
  represents a configuration $\sum_i\vec z_i\otimes x_i$ in $C(E;\bmX)$, where
  $\sigma^*\vec z_i\otimes x_i = \vec z_i\otimes \sigma_*x_i$ for
  each $\sigma\in\frS_{k_i}$.
\end{defi}

\begin{rem}\label{rem:labEffect}
  It is perhaps surprising how many different variations of these spaces
  can be produced by a suitable choice of the labelling sequence:
  \begin{enumerate}
  \item If all $X_k$ carry a trivial $\frS_k$-action, then $C(E;\bm{X})$
    contains unordered collections of labelled and internally unordered
    clusters. For example, let $\ul{\S}^0$ be the sequence with the $0$-sphere
    $\S^0$, together with trivial $\frS_k$-actions, in each degree. Then we
    have\looseness-1
    \[\textstyle C(E;\ul{\S}^0)\cong\coprod_{[K]} C_K(E).\]
  \item For $k\ge 1$ and a based space $X$ with a based $\frS_k$-action, let
    $X[k]\coloneqq (X_{l})_{l\ge 1}$ be the sequence with $X_k\coloneqq X$ and
    $X_l\coloneqq *$ for $l\ne k$. Then $C(E;X[k])$ contains only
    configurations where all clusters have size $k$. If the $\frS_k$-action is trivial,
    then $C(E;X[k])$ is exactly the space $C^k(E;X)$ from the
    introduction. In particular, we have\looseness-1
    \[\textstyle C(E;\S^0[k])=C^k(E;\S^0)\cong
      \coprod_{r\ge 0} C_{\smash{r\times k}}(E).\]    
  \item If we define $\frS_+$ to be the based symmetric sequence with
    $(\frS_+)_k\coloneqq \{*\}\sqcup\frS_k$, together with the left translation,
    then $C(E;\frS_+)$ contains \emph{unordered}
    collections of unlabelled, but internally \emph{ordered} clusters.
  \item For a based space $X$, let $X^\wedge$ be the based symmetric sequence with
    $(X^\wedge)_k\coloneqq X^{\wedge k}$, with $\frS_k$ acting by coordinate
    permutation. Then $C(E;X^\wedge)$ contains configurations of clusters where
    each \emph{particle} inside a cluster carries a label in $X$, and if
    \emph{one} of these labels reaches the basepoint, then the entire cluster vanishes. 
  \end{enumerate}
\end{rem}

We finish this section by formally defining the action of $\scC_d$ on $C(\R^d;\bmX)$ by inserting
configurations of labelled clusters into boxes as in \tref{Figure}{fig:EdonC}.

\begin{figure}[h]
  \centering
  \begin{tikzpicture}[scale=2.4]
  \node at (-.3,.5) {$\lambda_3$};
  \node at (-.1,.5) {$\left(\vbox to 1.5cm{}\right.$};
  \node at (4.1,.5) {$\left.\vbox to 1.5cm{}\right)$};
  \draw[thin,dgrey] (0,0) rectangle (1,1);
  \draw[thick] (.35,.1) rectangle (.55,.3);
  \draw[thick] (.5,.4) rectangle (.9,.8);
  \draw[thick] (.1,.6) rectangle (.4,.9);
  \node at (.7,.6) {\scriptsize $1$};
  \node at (.25,.75) {\scriptsize $3$};
  \node at (.45,.2) {\scriptsize $2$};
  \draw[thin,dgrey] (1.2,.1) rectangle (2,.9);
  \draw[thin,dgrey] (2.2,.1) rectangle (3,.9);
  \draw[thin,dgrey] (3.2,.1) rectangle (4,.9);
  \node at (4.31,0.5) {$=$};
  \draw[thin,dgrey] (4.5,0) rectangle (5.5,1);
  \draw[very thin,grey] (4.85,.1) rectangle (5.05,.3);
  \draw[very thin,grey] (5,.4) rectangle (5.4,.8);
  \draw[very thin,grey] (4.6,.6) rectangle (4.9,.9);
  \draw[black!40,thick] (1.55,.75) -- (1.5,.25);
  \draw[black!40,thick] (1.45,.55) -- (1.6,.4);
  \draw[black!40,thick,fill=black!10] (1.7,.6) -- (1.85,.4) -- (1.75,.3) -- (1.7,.6);
  \fill (1.55,.75) circle (.02);
  \fill (1.45,.55) circle (.02);
  \fill (1.6,.4) circle (.02);
  \fill (1.5,.25) circle (.02);
  \fill (1.8,.7) circle (.02);
  \fill (1.7,.6) circle (.02);
  \fill (1.85,.4) circle (.02);
  \fill (1.75,.3) circle (.02);
  \node at (1.1,.12) {,};
  \node at (2.1,.12) {,};
  \node at (3.1,.12) {,};
  \draw[black!40,thick] (3.5,.3) -- (3.5,.6);
  \fill (3.5,.6) circle (.02);
  \fill (3.5,.45) circle (.02);
  \fill (3.5,.3) circle (.02);
  \draw[black!40,thick] (5.175,.725) -- (5.15,.475);
  \draw[black!40,thick] (5.125,.625) -- (5.2,.55);
  \draw[black!40,thick,fill=black!10] (5.325,.55) -- (5.275,.5) -- (5.25,.65) -- (5.325,.55);
  \fill (5.175,.725) circle (.014);
  \fill (5.2,.55) circle (.014);
  \fill (5.125,.625) circle (.014);
  \fill (5.15,.475) circle (.014);
  \fill (5.3,.7) circle (.014);
  \fill (5.25,.65) circle (.014);
  \fill (5.325,.55) circle (.014);
  \fill (5.275,.5) circle (.014);
  \draw[black!40,thick] (4.7125,.675) -- (4.7125,.7875);
  \fill (4.7125,.7875) circle (.014);
  \fill (4.7125,.73125) circle (.014);
  \fill (4.7125,.675) circle (.014);
\end{tikzpicture}
  \caption{An instance of $\lambda_3\colon \scC_2(3)\times C(\R^2;\ulp{\S}^0)^3
    \to C(\R^2;\ulp{\S}^0)$}\label{fig:EdonC}
\end{figure}

\begin{constr}\label{constr:clustEd}
  Let $\bm{X}=(X_k)_{k\ge 1}$ be a symmetric sequence as before and let $d\ge 1$. Then
  $C(\R^{d};\bm{X})$ admits the structure of a $\smash{\scC_{d}}$-algebra:
  recall that operations in $\smash{\scC_{d}(s)}$ are tuples $(c_1,\dotsc,c_s)$
  of rectilinear embeddings
  $\smash{c_{\smash j}\colon [0;1]^{d}\hookrightarrow [0;1]^{d}}$ with pairwise disjoint
  image. If pick an iden\-tification $\R\cong (0;1)$, and thus, by coordinate-wise
  application, also $\R^{d}\cong (0;1)^{d}$, then our structure maps
  $\lambda_s\colon \scC_{d}(s)\times C(\R^d;\bm{X})^s\to C(\R^{d};\bm{X})$
  are given by\looseness-1
  \[\lambda_s\mathopen{}\bigg((c_1,\dotsc,c_s),
      \sum_{i=1}^{r_1}\vec z_{1,i}\otimes x_{1,i},
      \dotsc,
      \sum_{i=1}^{r_s}\vec z_{s,i}\otimes x_{s,i}
      \bigg)\mathclose{}
      \coloneqq \sum_{j=1}^s\sum_{i=1}^{\vphantom{J_g}r_{\kern-.4px\smash j}}
      c_j(\vec z_{j\kern-.3px,\kern.3pxi})\otimes x_{j\kern-.3px,\kern.3pxi}.\]
\end{constr}

\section{Cellular decompositions of clustered configuration spaces}
\label{sec:cell}
In this section, we study the homotopy type of the bar construction
$BC(\R;\bmX)$. After having introduced the necessary combinatorics, we start by
discussing the instructive example of $C(\R;\ul{\S}^0)$, and then establish a
strategy for the general case.

\begin{defi}
  For each integer $n\ge 0$, a \emph{partition} of $\{1,\dotsc,n\}$
  is a tuple \mbox{$\xi=(S_1,\dotsc,S_r)$} of non-empty subsets
  $S_i\subseteq \{1,\dotsc,n\}$ such that:\vspace*{4pt}
  \begin{enumerate}[packed,joinedup]
  \item the collection $\{S_1,\dotsc,S_r\}$ is a partition of
    $\{1,\dotsc,n\}$;\vspace*{1pt}
  \item the entries are ordered by their minimum,
    i.e.\ $\min(S_1)<\dotsb<\min(S_r)$.\vspace*{4pt}
  \end{enumerate}
  \noindent We write $\abs{\xi}\coloneqq n$ and
  $K(\xi)\coloneqq (\# S_1,\dotsc,\# S_r)$. Let $\Xi$ be the set of all
  partitions for all $n$.
\end{defi}

\begin{constr}\label{def:entType}
  We have a product $\Xi\times \Xi\to \Xi$ by \emph{stacking} partitions: more
  precisely, for two partitions
  $\xi=(S_{\smash 1}^{\vphantom\prime},\dotsc,S_{\smash r}^{\vphantom \prime})$
  and $\xi'=(S_{\smash 1}',\dotsc,S'_{\smash{r'}})$, we let
  \[\xi\sqcup \xi' \coloneqq \pa{S_{\smash 1}^{\vphantom\prime},\dotsc,S_{\smash
        r}^{\vphantom\prime},|\xi|+S_{\smash 1}',\dotsc, |\xi|+S'_{\smash{r'}}}.\]
  Thus, $\Xi$ becomes a monoid with neutral element the empty partition
  $\emptyset$. This monoid is free: we call a partition
  $e\in \Xi$ \emph{indecomposable}, or an \emph{entanglement type}, if it is neither
  empty nor the product of two non-empty partitions, and we denote\footnote{As a
    typographical mnemonic, $\Xi$ looks like a ‘decomposable’ version of $\bE$.}
  the subset of them by $\bE\subseteq \Xi$. Then the monoid $\Xi$ is freely
  generated by $\bE$.
  This generating set $\bE$ is graded: for an entanglement type $e=(S_1,\dotsc,S_w)$
  we let $\# e\coloneqq w$ be its \emph{weight}.
\end{constr}

\begin{expl}\label{ex:CRS0}
  We have a map $\chi\colon C(\R;\ul{\S}^0)\to \Xi$ of $E_1$-algebras given by
  identifying, for each $\sum_i [\vec z_i]\in C_K(\R)$, the set
  $\bigcup_i[\vec z_i]\subseteq \R$ with $\{1,\dotsc,|K|\}$ in a monotone way,
  and regard clusters as entries of the partition, see
  \tref{Figure}{fig:vr01s0}.
  
  This map admits a section $s$ by including $\{1,\dotsc,|K|\}$ into $\R$, and
  the composition $s\circ\chi$ is homotopic to the identity by linear
  interpolation. Thus, $\chi$ is an equivalence of $E_1$-algebras.  Since $\Xi$
  is a freely generated by $\bE$, we get
  $BC(\R;\ul{\S}^0)\simeq\textstyle\bigvee_{e\in\bE}\S^1$.
\end{expl}

\begin{tFigure}
  \centering
  \begin{tikzpicture}[scale=3.5]
  \draw[thin,dgrey] (.1,0) -- (1.1,0);
  \draw[dgrey,thick] (.2,0) -- (.2,.08) -- (.6,.08) -- (.6,0);
  \draw[dgrey,thick] (.6,.08) -- (.8,.08) -- (.8,0);
  \draw[dgrey,thick] (.25,0) -- (.25,.04) -- (.4,.04) -- (.4,0);
  \draw[dgrey,thick] (.35,0) -- (.35,-.06) -- (.9,-.06) -- (.9,0);
  \draw[dgrey,thick] (.94,0) -- (.94,.04) -- (1,.04) -- (1,0);
  \draw[black!20,thin] (.05,-.14) rectangle (1.15,.18);
  \draw[black!20,thin] (1.43,-.14) rectangle (2.48,.18);
  \node at (.2,0) {\tiny $\bullet$};
  \node at (.25,0) {\tiny $\bullet$};
  \node at (.35,0) {\tiny $\bullet$};
  \node at (.4,0) {\tiny $\bullet$};
  \node at (.7,0) {\tiny $\bullet$};
  \node at (.85,0) {\tiny $\bullet$};
  \node at (.7,0) {\tiny $\bullet$};
  \node at (.9,0) {\tiny $\bullet$};
  \node at (.94,0) {\tiny $\bullet$};
  \node at (1,0) {\tiny $\bullet$};
  \node at (1.29,0) {\small $\longmapsto$};
  \node[anchor=south] at (1.5,-.08) {\tiny $\strut 1$};
  \node[anchor=south] at (1.6,-.08) {\tiny $\strut 2$};
  \node[anchor=south] at (1.7,-.08) {\tiny $\strut 3$};
  \node[anchor=south] at (1.8,-.08) {\tiny $\strut 4$};
  \node[anchor=south] at (1.9,-.08) {\tiny $\strut 5$};
  \node[anchor=south] at (2.0,-.08) {\tiny $\strut 6$};
  \node[anchor=south] at (2.1,-.08) {\tiny $\strut 7$};
  \node[anchor=south] at (2.2,-.08) {\tiny $\strut 8$};
  \node[anchor=south] at (2.3,-.08) {\tiny $\strut 9$};
  \node[anchor=south] at (2.4,-.08) {\tiny $\strut 10$};
  \draw[dgrey,thick] (1.5,.04) -- (1.5,.12) -- (1.9,.12) -- (1.9,.04);
  \draw[dgrey,thick] (1.9,.12) -- (2.1,.12) -- (2.1,.04);
  \draw[dgrey,thick] (1.6,.04) -- (1.6,.08) -- (1.8,.08) -- (1.8,.04);
  \draw[dgrey,thick] (1.7,-.04) -- (1.7,-.08) -- (2.2,-.08) -- (2.2,-.04);
  \draw[dgrey,thick] (2.3,.04) -- (2.3,.08) -- (2.4,.08) -- (2.4,.04);
  \node[anchor=west] at (2.49,0) {\footnotesize $=\pa{\{1,5,7\},\{2,4\},\{3,8\},\{6\},\{9,10\}}$};
\end{tikzpicture}
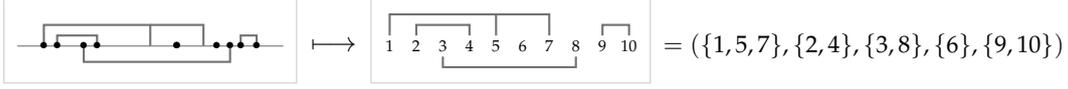
  \caption{An instance of $\chi\colon C(\R;\ulp{\S}^0)\to \Xi$}\label{fig:vr01s0}
\end{tFigure}

In the case of general labelling sequences $\bmX=(X_k)_{k\ge 1}$ with
non-isolated basepoints, we have to deal with the phenomenon that new clusters
can suddenly arise or vanish when a label leaves or enters the basepoint,
respectively. In order to gain control ‘near’ the basepoint, we will have to
assume that each $X_k$ is \emph{well-based}, i.e.\ the basepoint inclusion
$*\hookrightarrow X_k$ is a cofibration in the Quillen model structure of spaces
(it is not necessary to consider cofibrations of $\frS_k$-spaces; see
Remark~\ref{rem:noSym} for a conceptual reason.)

\begin{theo}\label{thm:BCRX}
  Let $\bm{X}=(X_k)_{k\ge 1}$ be a sequence of well-based
  spaces (with arbitrary based $\frS_k$-actions on $X_k$). Then we
  have a weak equivalence, abbreviating
  $\bmX^{\wedge K}\coloneqq X_{k_1}\wedge \dotsb\wedge X_{k_r}$,\looseness-1
  \[\textstyle BC(\R;\bm{X})\simeq \Sigma\bigvee_{e\in
      \bE}\bm{X}^{\wedge K(e)}.\vspace*{2px}\]
\end{theo}

\begin{expl}In many special cases, Theorem~\ref{thm:BCRX} has an easier shape:
  \begin{enumerate}
  \item For $\bmX=\ul{\S}^0$, this is precisely \tref{Example}{ex:CRS0}.
  \item Given $k\ge 1$ and a well-based space $X$, endowed with the trivial
    $\frS_k$-action, the case of $\bmX\coloneqq X[k]$ recovers
    Theorem~\ref{thm:a}. Note that only entanglement types $e$ with
    $K(e)=r\times k$ for some $r$ are relevant here, since
    $X[k]^{\wedge K(e)}=*$ otherwise.
  \item If $X$ is a well-based space, then the case of $\bmX\coloneqq X[1]$
    recovers Segal’s result: if $K\ne (1,\dotsc,1)$, then
    $X[1]^{\smash{\wedge K}}=*$, but there is only one entanglement type
    involving only singletons, namely $(\{1\})$. Thus, we get
    $BC(\R;X)=BC(\R;X[1])\simeq \Sigma X$.
  \end{enumerate}
\end{expl}

\begin{rem}\label{rem:noSym}
  The reader should not be surprised by the fact that the symmetric
  actions on $\bm{X}$ do not appear on the right side—they are also irrelevant
  for the left side: for each tuple $K=(k_1,\dotsc,k_r)$, the action of
  $\prod_i\frS_{\smash{k_i}}$ on $\tilde{C}_K(\R)$ induces a free action on $\pi_0$, so
  we can alternatively restrict to the subspace $\tC^<_K(\R)$ containing
  configurations $(\vec z_1,\dotsc,\vec z_r)$ of clusters where each cluster
  $\vec z_i=(z_{i,1},\dotsc,z_{i,k_i})$ satisfies $z_{i,1}<\dotsb<z_{i,k_i}$ in
  $\R$: if we write $\mathbf{1}\coloneqq (1)_{k\ge 0}$ for the sequence of trivial
  groups, then we get the equivalent description\vspace*{-1px}
    \[C(\R;\bm{X}) \cong \int^{K\in \mathbf{1}\swr\mathbf{Inj}}\tC^<_K(\R)\times\bm{X}^K.\]
\end{rem}

We prove \tref{Theorem}{thm:BCRX} by decomposing $C(\R;\bmX)$
into \emph{free} $E_1$-algebras as follows:

\begin{constr}
  For each integer $w\ge 0$, let $\scF_w\Xi\subseteq \Xi$ be the
  submonoid generated by all entanglement types of weight at most $w$. This
  gives rise to a filtration, which is exhaustive since $e\in \scF_{\# e}\Xi$.
  
  Using the map $\chi\colon \coprod_{[K]}C_K(\R)\to \Xi$ from
  \tref{Example}{ex:CRS0}, we construct an exhaustive filtration of $C(\R;\bmX)$
  for each based symmetric sequence $\bmX$ by defining
  \[\textstyle \scF_wC(\R;\bmX)\coloneqq
    \mathopen{}\Big\{\hspace*{-1.5px}\sum_i
    \vec z_i\otimes x_i;\,\text{all $x_i\ne *$ and }\chi(\sum_i[\vec z_i])\in
    \scF_w\Xi\hspace*{-.5px}\Big\}\mathclose{}.\]
  Since $\chi$ is a map of $E_1$-algebras, each $\scF_wC(\R;\bmX)$ is an
  $E_1$-subalgebra of $C(\R;\bmX)$, and since the bar construction commutes with
  filtered colimits, we recover $BC(\R;\bmX)$ as the direct limit of the spaces
  $B\scF_\bullet C(\R;\bmX)$.
\end{constr}

Visually, given a labelled configuration inside $C(\R;\bmX)$, two clusters with
non-trivial label are ‘entangled’ if their convex hulls on the real line
intersect, see \tref{Figure}{fig:vr01s0}, and each equivalence class with
respect to this relation determines an entanglement type (in
\tref{Figure}{fig:vr01s0}, there are two equivalence classes, with weights $4$
and $1$). Then $\scF_wC(\R;\bmX)$ contains all configurations for which only
entanglement types of weight at most $w$ occur.\looseness-1

The main part of the proof of \tref{Theorem}{thm:BCRX} is to see that
$\scF_wC(\R;\bmX)$ is equivalent to an $E_1$-algebra that arises from
$\scF_{w-1}C(\R;\bmX)$ by attaching a free $E_1$-algebra. Let us first establish
the notion of an $E_1$-cell attachment, which is inspired by
\cite{Galatius-Kupers-RW-2018}.

\begin{constr}
  If $\scO$ is an operad with $\scO(0)=\{*\}$, then each $\scO$-algebra has an
  underlying based space. The forgetful functor $U$ to based spaces has a left adjoint,
  called $F$. Explicitly, $FX$ is given by quotienting
  $\coprod_{r\ge 0}\scO(r)\times_{\frS_r}X^r$ by the basepoint relations from
  §\,\ref{sec:introduction}.\looseness-1

  For a map $\imath\colon A\to Y$ of based spaces, an $\scO$-algebra $M$,
  and a based map \mbox{$g\colon A\to UM$}, we define the \emph{$\scO$-cell attachment}
  $\smash{M\sqcup^{\hspace*{1px}\smash{\scO}}_A Y}$ as the pushout of $\scO$-algebras
  \[\begin{tikzcd}
      FA\ar[r,"\bar g"]\ar{d}[swap]{F\imath}\arrow[dr, phantom, "\ulcorner", very near end] & M\ar[d]\\
      FY\ar[r] & M\sqcup^{\hspace*{1px}\smash{\scO}}_A Y,
    \end{tikzcd}\]
  where $\bar g$ is the adjoint of $g$. If $T\coloneqq UF$ denotes the monad
  associated with $\scO$, then $\smash{M\sqcup^{\hspace*{1px}\smash{\scO}}_A Y}$
  is the reflexive coequaliser (in $\scO$-algebras, as well as in based spaces)
  of
  \begin{align}\label{eq:coeq}
    F(TUM\sqcup_{A}Y)\rightrightarrows F(UM\sqcup_AY).
  \end{align}
  Here $UM\sqcup_A Y$ and $TUM\sqcup_A Y$ are pushouts of based spaces, the first
  arrow of (\ref{eq:coeq}) is induced by the action $TUM\to UM$, the second arrow is given by
  applying $F$ to the inclusion $TUM\sqcup_{A} Y\to T(UM\sqcup_A M)$ and
  com\-posing with the counit $FT=FUF\Rightarrow F$, and the degeneracy is
  induced by the unit $UM\to TUM$, see
  \cite[§\,6.1]{Galatius-Kupers-RW-2018}.\looseness-1
\end{constr}

\begin{expl}\label{ex:spaceE1push}
  Let us unravel the above construction in two cases:
  \begin{enumerate}
  \item Restricting to one model, an \emph{$E_1$-algebra} is the same as
    an algebra over $\scC_1$. If $M$ is an $E_1$-algebra
    and $\imath\colon A\to Y$ and $f\colon A\to UM$ are based maps, then
    points in $\smash{M\sqcup^{\smash{E_1}}_A Y}$ are given by
    configurations of disjoint subintervals
    \mbox{$c_1,\dotsc,c_s\colon [0;1]\hookrightarrow [0;1]$}, each carrying a label
    in $UM\sqcup_A Y$, quotiented by the usual basepoint
    relation; and additionally, if $c_i$ is labelled by
    $\lambda_t(c'_{\smash 1},\dotsc,c'_{\smash t};m_1,\dotsc,m_t)\in M$ with
    $(c'_1,\dotsc,c_t')\in\scC_1(t)$, $m_1,\dotsc,m_t\in M$, and
    $\lambda_t\colon \scC_1(t)\times M^{\smash t}\to M$ being the $\scC_1$-action, then the
    configu\-ration is identified with the one where the interval $c_i$ is replaced
    by the intervals
    $c_{\smash i}^{\vphantom\prime}\circ c_{\smash 1}',\dotsc,c_{\smash
      i}^{\vphantom\prime}\circ c_{\smash t}'$, carrying the labels
    $m_1,\dotsc,m_t$, respectively.\looseness-1
  \item Algebras over the \emph{associative operad} are the same as topological
    monoids. If $M$ is a topological monoid and $\imath\colon A\to Y$ and
    $f\colon A\to UM$ are based maps, then points in
    $\smash{M\sqcup^{\smash{\text{Mon}}}_A Y}$ are given by
    \emph{strings} $\zeta_1\dotsm\zeta_r$ with $\zeta_i\in UM\sqcup_A Y$.
    If $\zeta_i$ is the basepoint, then it can be omitted from the string,
    and if $\zeta_i,\zeta_{i+1}\in UM$, then the substring $\zeta_i\zeta_{i+1}$
    can be replaced by the single letter that equals the actual product $\zeta_i\cdot \zeta_{i+1}$.
  \end{enumerate}
\end{expl}

\begin{rem}\label{rem:E1homo}
  For our purposes, it is convenient to have a homotopically better behaved
  construction: the reflexive pair in (\ref{eq:coeq}) is part of an entire
  simplicial $\scO$-algebra\looseness-1
  \[P^\scO_\bullet(M,A,Y)\colon [n]\mapsto F(T^nUM\sqcup_{A} Y).\]
  Its geometric realisation is denoted by
  $M\cup^{\smash{\hspace*{1px}\scO}}_A Y$ and called the \emph{derived
    $\scO$-cell attachment}, see \cite[§\,3.1]{Kupers-Miller},
  \cite[§\,8.3.6]{Galatius-Kupers-RW-2018}.  We have an augmentation
  \mbox{$P^\scO_\bullet(M,A,Y)\to M\sqcup^{\smash{\hspace*{1px}\scO}}_A Y$},
  inducing a map
  $M\cup^{\smash{\hspace*{1px}\scO}}_A Y\to
  M\sqcup^{\smash{\hspace*{1px}\scO}}_A Y$ of $\scO$-algebras.  Therefore, maps
  out of the derived attachment into another $\scO$-algebra can equally well be
  declared on $M$ and $Y$.\looseness-1

  If $\imath\colon A\to Y$ is a cofibration between well-based spaces (in the
  Quillen model structure of spaces) and if $M$ is cofibrant (in the projective
  model structure on $\scO$-algebras), then the above map
  $M\cup^{\smash{\hspace*{1px}\scO}}_A Y\to M\sqcup^{\smash{\hspace*{1px}\scO}}_A Y$
  is a weak equivalence, compare \cite[§\,8.2]{Galatius-Kupers-RW-2018}: this
  reflects the fact that under these conditions, the \emph{actual} pushout is a
  \emph{homotopy} pushout.\looseness-1

  In the case of $E_1$-algebras, it follows\footnote{To be precise,
    \cite[Prop.\,98]{Kupers-Miller} only treats the case where $(Y,A)$ is a disc
    $(\bbD^n,\S^{n-1})$. However, the proof goes through for the general case
    without any modifications.} from \cite[Prop.\,98]{Kupers-Miller} that the
  bar construction $\smash{B(M\cup^{\smash{E_1}}_A Y)}$ arises from $BM$ by
  attaching $\Sigma Y$ along the map $\Sigma A\to \Sigma UM\to BM$, i.e.\
  the bar construction turns derived $E_1$-attachments into suspended attachments.
\end{rem}

After this general interlude, let us come back to the configuration spaces $C(\R;\bmX)$.
\enlargethispage{\baselineskip}

\begin{defi}\label{def:degtup}
  Let $\bmX=(X_k)_{k\ge 1}$ be a sequence of well-based spaces and
  \mbox{$K=(k_1,\dotsc,k_r)$} be a tuple of positive integers. Then we define
  \[\bm{X}^{\Delta K} \coloneqq \set{(x_1,\dotsc,x_r)\in \bm{X}^K;\,x_i=*_{k_i}
      \text{ for some $i$}}\subseteq \bm{X}^K\]
  as the subspace of \emph{degenerated tuples}, with basepoint
  $(*_{k_1},\dotsc,*_{k_r})$. Note that since each $X_k$ is assumed to be
  well-based, $\bmX^{\Delta K}\hookrightarrow \bmX^K$ is a cofibration of
  well-based spaces.
\end{defi}

\begin{constr}\label{constr:attach}
  For each entanglement type $e$ of weight $w$, we have a based map
  $\tilde f_e\colon \bmX^{K(e)}\to \scF_{w}C(\R;\bmX)$ defined as follows: if we
  write $e=(S_1,\dotsc,S_w)$ and include the set $\{1,\dotsc,|e|\}$ canonically
  into $\R$, then each subset $S_i$, together with the order inherited from
  $\R$, can be regarded as an ordered cluster $\vec z_i$, and we put
  \[\tilde f_e(x_1,\dotsc,x_w)\coloneqq
    \sum_{i=1}^w \vec z_i\otimes x_i\in \scF_w C(\R;\bmX).\]
  If $(x_1,\dotsc,x_w)$ lies in the subspace $\bmX^{\Delta K(e)}$, then
  the labelled configuration $\tilde f_e(x_1,\dotsc,x_w)$ has at most $w-1$ non-trivial
  clusters, and thus, the restriction $f_e$ of $\tilde f_e$ to $\bmX^{\Delta K(e)}$
  lands in the filtration component $\scF_{w-1}C(\R;\bmX)$. If we use the
  bouquet $\bigvee_{\#e =w}f_e$ to $E_1$-attach $\bigvee_{\# e=w}\bmX^{K(e)}$ to
  $\scF_{w-1}C(\R;\bmX)$ in the derived sense, then the extensions $\tilde f_e$
  declare, via the universal property, a map of $E_1$-algebras under
  $\scF_{w-1}C(\R;\bmX)$ of the form
  \begin{align}\label{eq:phi}
    \textstyle\phi_w\colon \scF_{w-1}C(\R;\bmX)
    \cup^{\smash{E_1}}_{\bigvee_e\hspace*{-.5px}\bmX^{\Delta K(e)}}
    \bigvee_{\# e=w} \bmX^{\smash{K(e)}}\longrightarrow \scF_wC(\R;\bmX).
  \end{align}
\end{constr}

\begin{lem}\label{lem:eqAtt}
  The map $\phi_w$ is an equivalence of $E_1$-algebras.
\end{lem}

This shows that, up to equivalence, $C(\R;\bmX)$ can inductively be built by
attaching free $E_1$-algebras. We first prove \tref{Theorem}{thm:BCRX} using
\tref{Lemma}{lem:eqAtt}, and then prove the Lemma.

\begin{proof}[Proof of \tref{Theorem}{thm:BCRX}]
  Let us abbreviate $C\coloneqq C(\R;\bmX)$.  Using that each $X_k$ is
  well-based, the inclusions \mbox{$U\scF_{w-1}C\hookrightarrow U\scF_wC$}, and
  hence also the inclusions $B\scF_{w-1}C\hookrightarrow B\scF_w C$, are
  (Hurewicz) cofibrations of spaces.  Therefore, $BC$ is equivalent to the
  \emph{homotopy} colimit over the filtration components $B\scF_\bullet C$.
  Since $B\scF_0C$ is just a point, it suffices to show that the induced map
  \mbox{$B\scF_{w-1}C\to B\scF_wC$} is equivalent to the inclusion into the
  bouquet
  \mbox{$B\scF_{w-1}C\hookrightarrow B\scF_{w-1}C\vee \Sigma\bigvee_{\# e=w}
    \bmX^{\wedge K(e)}$} for each $w$.\looseness-1

  This equivalence is established in two steps: first, we use the equivalence
  $\phi_w$ from \tref{Lemma}{lem:eqAtt}. Again, since each $X_k$ is well-based,
  the induced map $B\phi_w$ is a weak equivalence of based spaces (the map
  $B_\bullet(\Sigma,T^{\smash{E_1}},\phi_w)$ among the two-sided bar
  constructions is a levelwise equivalence of proper simplicial spaces). As the
  bar construction turns $E_1$-attachments into suspended attachments
  (see Remark~\ref{rem:E1homo}), we get a homotopy pushout\looseness-1
  \[
    \begin{tikzcd}
      \Sigma\bigvee_{\# e=w} \bmX^{\Delta K(e)}\ar[r]\ar[d,hook]\arrow[dr, phantom, "\ulcorner", very near end] & B\scF_{w-1}C\ar[d]\\
      \Sigma\bigvee_{\# e=w} \bmX^{K(e)} \ar[r] & B\scF_w C.
    \end{tikzcd}
  \]
  Second, we consider the left vertical map: by elementary homotopy theory, the
  cofibre sequence $\bmX^{\Delta K(e)}\to \bmX^{K(e)}\to \bmX^{\wedge K(e)}$
  splits after a single suspension for each $e$. Thus, each of the
  summands in the left vertical map above is equivalent to the wedge inclusion
  \mbox{$\Sigma\bmX^{\Delta K(e)}\hookrightarrow \Sigma\bmX^{\Delta K(e)}\vee
    \Sigma\bmX^{\wedge K(e)}$}. As the attaching map is, for each $e$, defined
  on the \emph{first} of the two wedge summands, the attachment is the same as
  adding the \emph{second} one.\looseness-1
\end{proof}

\clearpage

\newcommand{\XDel}{\smash{\bmX}^{\smash{\Delta}}_{\smash{w}}\vphantom{X^r}}
\newcommand{\xDel}{\smash{\bmX}^{\smash{\Delta}}_{\smash{w}}\vphantom{X}}

\begin{proof}[Proof of \tref{Lemma}{lem:eqAtt}]
  Recall that we have to show that the map $\phi_w$ from (\ref{eq:phi}) is a
  weak equivalence. First, we simplify our notation: as before, we write
  $C\coloneqq C(\R;\bmX)$; and additionally, let
  $\tilde f_w\coloneqq \bigvee_{\# e=w}\tilde f_e$,
  $f_w\coloneqq\bigvee_{\# e=w}f_e$, $\bmX_w\coloneqq \bigvee_{\# e=w}\bmX^{K(e)}$, and
  ${\bmX_w^\Delta}\coloneqq \bigvee_{\# e=w}\bmX^{\Delta K(e)}$.\looseness-1

  Now the proof strategy is to ‘discard’ contractible information on both sides
  of $\phi_w$ by introducing a ‘thin’ version $D$ of $C$, which is even a
  topological monoid, and which comes with a filtration $\scF_\bullet D$ by
  submonoids. The proof then proceeds as follows:
  \begin{enumerate}
  \item Construct the topological monoid $D$ and its filtration
    $\scF_\bullet D$, and construct equivalences
    $\rho_\bullet\colon \scF_\bullet C\to \scF_\bullet D$ of $E_1$-algebras,
    which commute with the inclusions.
  \item We use $\rho_{w-1}f_w\colon {\bmX_w^\Delta}\to \scF_{w-1}D$ to attach $\bmX_w$ to
    $\scF_{w-1}D$. Show that the induced map
    $\smash{\rho_{w-1}\cup^{\smash{E_{\smash 1}}}_{\XDel}\bmX_w\colon
      \scF_{w-1}C\cup^{\smash{E_{\smash 1}}}_{\XDel} \bmX_w\to \scF_{w-1}D\cup^{\smash{E_{\smash 1}}}_{\XDel}
      \bmX_w}$ is an equivalence.\looseness-1
  \item Let $\alpha\colon T^{\smash{E_1}}\Rightarrow T^{\text{Mon}}$ be the
    transformation of monads. Show that the induced map
    $\scF_{w-1}D\cup^{\hspace*{1px}\alpha}_{\xDel} \bmX_w\colon
    \scF_{w-1}D\cup^{\smash{E_{\smash 1}}}_{\XDel}\bmX_w\to
    \scF_{w-1}D\cup^{\smash{\text{Mon}}}_{\xDel}\bmX_w$ is an equivalence.
  \item Show that the map
    $\smash{\psi_w\colon \scF_{w-1}D\cup^{\smash{\text{Mon}}}_{\xDel}\bmX_w\to \scF_wD}$
    that is, via the universal property, induced by
    $\smash{\rho_w\tilde f_w\colon \bmX_w\to \scF_w D}$ is an equivalence.
  \end{enumerate}
  \noindent Since $\phi_w$ is induced by $\tilde f_w\colon \bmX_w\to \scF_w C$ and
  $\psi_w$ is induced by $\rho_w\tilde f_w\colon \bmX_w\to \scF_w D$, the above maps
  assemble into a commutative square
  \begin{align}
    \begin{tikzcd}[ampersand replacement=\&,column sep=5em]
      \scF_{w-1}C\cup^{\smash{E_{\smash 1}}}_{\XDel} \bmX_w\ar[r,"\phi_w"]\ar{d}{\simeq}[swap]{\rho_{w-1}\cup^{\smash{E_{\raisebox{1pt}{\scaleto{1}{3pt}}}}}_{{\smash{\bmX}_{\raisebox{1px}{\scaleto{w}{2pt}}}^{\hspace*{-.1px}\scaleto{\Delta}{3pt}}}} \bmX_w}
      \& \scF_w C\ar{dd}{\rho_w}[swap]{\simeq}\\
      \scF_{w-1}D\cup^{\smash{E_{\smash 1}}}_{\XDel}\bmX_w\ar{d}[swap]{\scF_{w-1}D\hspace*{.5px}\cup^{\hspace*{1px}\alpha}_{{\smash{\bmX}_{\raisebox{1px}{\scaleto{w}{2pt}}}^{\hspace*{-.1px}\scaleto{\Delta}{3pt}}}} \bmX_w}{\simeq} \&\\
      \scF_{w-1}D\cup^{\smash{\text{Mon}}}_{\xDel}\bmX_w\ar{r}{\psi_w}[swap]{\simeq} \& \scF_wD.
    \end{tikzcd}
  \end{align}
  It then follows from the 2-out-of-3-property that the map $\phi_w$ in question
  is a weak equivalence, which finishes the proof. Let us go through the steps
  1\,–\,4:\looseness-1
  
  \medskip
  
  1.\hspace{.5em}Replacing $\tC_K(\R)$ by its set of path components, we define
  $D \coloneqq \int^{K}\pi_0\tilde C_K(\R)\times \bmX^K$.  Then elements in $D$
  are equivalence classes $[\xi;x_1,\dotsc,x_r]$, where $\xi=(S_1,\dotsc,S_r)$
  is a partition and where $x_i\in X_{\# S_i}$; and if $x_i$ is the basepoint,
  then $[\xi;x_1,\dotsc,x_r]$ is identified with
  $[d_i\xi;x_1,\dotsc,\widehat{x_i},\dotsc,x_r]$, where $d_i\xi$ arises from
  $\xi$ by removing $S_i$ and relabelling the remaining subsets. Defining
  $[\xi;x_{\smash 1}^{\vphantom\prime},\dotsc,x_{\smash
    r}^{\vphantom\prime}]\cdot [\xi';x'_{\smash
    1},\dotsc,x'_{\smash{r'}}]\coloneqq [\xi\sqcup \xi';x_{\smash
    1}^{\vphantom\prime},\dotsc,x_{\smash r}^{\vphantom\prime},x_{\smash
    1}',\dotsc,x_{\smash{r'}}']$, $D$ becomes a topological monoid, in
  particular an $E_1$-algebra. Moreover, $D$ is filtered by submonoids
  $\scF_wD\subseteq D$ containing only points that can be represented by
  $(\xi;x_1,\dotsc,x_r)$ where no $x_i$ is a basepoint and $\xi\in \scF_w\Xi$.
  We have, for each $w$, a map \mbox{$\rho_w\colon \scF_wC\to \scF_wD$} induced
  by the canonical projections $\tC_K(\R)\to \pi_0 \tC_K(\R)$.  This clearly is
  a morphism of $E_1$-algebras, and it commutes with the filtration in the sense that
  the (co-)restriction of $\rho_w$ to the $(w-1)$\textsuperscript{st} filtration
  level is precisely $\rho_{w-1}$. We show that each $\rho_w$ is a homo\-topy
  equivalence: since each $X_k$ is well-based, we find, for each $k\ge 1$, a
  map $u_k\colon X_k\to [0;1]$ satisfying
  $u_{\smash k}^{-1}(0)=\{*_k\}$. These can be used to construct a section $s_w$
  of $\rho_w$ by sending $[\xi;x_1,\dotsc,x_r]$ to a labelled configuration in
  $\R$, employing the unique inclusion
  $\nu\colon \{1,\dotsc,|K|\}\hookrightarrow \R$ with $\nu(1)=0$ and
  $\nu(j+1)-\nu(j)=u_{k_i}(x_i)$ for $j\in S_i$. Finally, the composition
  $s_w\circ\rho_w$ is homotopic to the identity by linear
  interpolation.\footnote{This is the usual argument showing that for a
    well-based space $X$, the classical labelled configuration space $C(\R;X)$
    is equivalent to the reduced James product over $X$.}\looseness-1

  \medskip
  
  2.\hspace*{.5em}Since the monad $T\coloneqq T^{\smash{E_1}}$ preserves
  well-based objects and equivalences between them, and since
  $\bmX_w^{\smash\Delta}\hookrightarrow \bmX_w^{\vphantom{\smash\Delta}}$ is a
  cofibration, the maps $\smash{T^\bullet\rho_{w-1}\sqcup_{\xDel} \bmX_w}$ are
  equivalences, and the same applies to
  $\smash{T(T^\bullet\rho_{w-1}\sqcup_{\xDel} \bmX_w)}$.  This shows that
  $UP^{\smash{E_1}}_\bullet(\rho_{w-1},\bmX_w^{\smash\Delta},\bmX_w^{\vphantom{\smash
      \Delta}})$ is a levelwise equivalence; finally, we use that the simplicial
  spaces on both sides are proper as the unit of the monad is a cofibration.

  \medskip

  3.\hspace*{.5em}We use that
  $\alpha\colon T^{\smash{E_1}}Y\to T^{\smash{\text{Mon}}}Y$ is a homotopy
  equivalence if $Y$ is well-based; this is just a variation of the above
  argument. This shows that the induced map
  $UP^{\smash{E_1}}_\bullet(\scF_{w-1}D,\bmX_w^{\smash\Delta},\bmX_w^{\vphantom{\smash\Delta}})\to
  UP^{\smash{\text{Mon}}}_\bullet(\scF_{w-1}D,\bmX_w^{\smash\Delta},\bmX_w^{\vphantom{\smash\Delta}})$ is a
  levelwise equivalence. Finally, we use again that both simplicial spaces
  are proper to obtain the substatement.

  \medskip
  
  4.\hspace*{.5em}We have to show that $\psi_w$ is a weak equivalence. To do so,
  we show that the map
  $\psi_w'\colon \scF_{w-1}D\sqcup_{\xDel}^{\smash{\text{Mon}}}\bmX_w\to
  \scF_wD$ from the \emph{strict} pushout is an \emph{isomorphism}.  \mbox{Since
    $\scF_0 D=*$}, this inductively shows that $\scF_{w-1}D$ is cofibrant in the
  projective model structure. As
  $\bmX_w^{\smash\Delta}\hookrightarrow \bmX_w^{\vphantom{\smash\Delta}}$ is a
  cofibration of well-based spaces,
  \mbox{$\smash{\scF_{w-1}D\cup_{\xDel}^{\smash{\text{Mon}}}\bmX_w\to
      \scF_{w-1}D\sqcup_{\xDel}^{\smash{\text{Mon}}}\bmX_w}$} is a weak
  equivalence, which then finishes the proof.  To show that $\psi'_w$ is an
  isomorphism, recall that points in
  $\scF_{w-1}D\sqcup^{\smash{\text{Mon}}}_{\xDel} \bmX_w$ are strings
  $\zeta_1\dotsm\zeta_s$ with letters $\zeta_i$ in the space
  $U\scF_{w-1}D\sqcup_{\xDel} \bmX_w$, identified by the relations from
  \tref{Example}{ex:spaceE1push}. Then the inverse of $\psi'_w$ is given as
  follows: each point $m\in \scF_w D$ can be written as
  \mbox{$[\xi;x_1,\dotsc,x_r]$} such that no $x_i$ is the respective
  basepoint. We can decompose \mbox{$\xi=e_1\sqcup\dotsb\sqcup e_s$} into
  entanglement types, i.e.\ $m=[e_1;\bm{x}_1]\dotsm [e_s;\bm{x}_s]$ with
  $\bm{x}_i\coloneqq (x_{w_1+\dotsb+w_{i-1}+1},\dotsc,x_{w_1+\dotsb+w_i})$ for
  $w_i\coloneqq \# e_i$.  If $w_i\le w-1$, then the factor $[e_i;\bm{x}_i]$
  already lies in $U\scF_{w-1}D$, and if $w_i=w$, then $[e_i,\bm{x}_i]$ can be
  regarded as an element in $\bmX_w$.  In this way, $m$ determines a string with
  letters in $U\scF_{w-1}D\sqcup_{\xDel}\bmX_w$ as above. One easily checks that
  this assignment factors through the relations for $\scF_w D$ and indeed forms
  an inverse of $\phi_w'$.\looseness-1
\end{proof}

\begin{cor}
  Let $\bmX=(X_k)_{k\ge 1}$ be a well-based sequence such that each $X_k$ is path-connected.
  Then $C(\R;\bmX)$ is equivalent to a free $E_1$-algebra.
\end{cor}
\begin{proof}
  Since each $X_k$ is path-connected, the $E_1$-algebra $C(\R;\bmX)$ is
  path-connected as well. Therefore, the canonical map
  $C(\R;\bmX)\to \Omega BC(\R;\bmX)$ is an equivalence. Now we
  use that by \tref{Theorem}{thm:BCRX}, $BC(\R;\bmX)$ is equivalent to
  $\Sigma \bigvee_e\bmX^{\wedge K(e)}$. Since $\bigvee_e\bmX^{\wedge K(e)}$
  is path-connected, this establishes an equivalence
  of $E_1$-algebras\looseness-1
  \[\textstyle C(\R;\bmX)\simeq \Omega BC(\R;\bmX)\simeq
    \Omega\Sigma\bigvee_e\bmX^{\wedge K(e)}\simeq
    F^{\smash{E_1}}\pa{\bigvee_e\bmX^{\smash{\wedge K(e)}}}.\qedhere\]
\end{proof}

\section{Iterated bar constructions of vertical configuration spaces}
\label{sec:segal}
While we understood the bar construction of the $E_1$-algebra $C(\R;\bmX)$ in
the previous section, the iterated bar construction of the $E_d$-algebra
$C(\R^d;\bmX)$ still has no geometric interpretation for $d\ge 2$. In this
section, we give a partial answer by introducing a family of subalgebras
$C(\R^{p\kern-.6px,\kern.6pxd-p};\bmX)\subseteq C(\R^d;\bmX)$ and studying their
$p$-fold bar construction.

As already motivated in the introduction, these subalgebras are constructed by
imposing a certain ‘verticality’ condition on particles within the same
cluster. Let us start by making this definition precise.

\begin{defi}\label{def:vert}
  Let $\pi\colon E\to B$ be a map of spaces. A cluster
  $\vec z=(z_1,\dotsc,z_k)$ in $E$ is called \emph{$\pi$-vertical}, if all
  particles $z_1,\dotsc,z_k$ lie in the same fibre.  For each tuple
  $K=(k_1,\dotsc,k_r)$, we let $\tC_K^\pi(E)\subseteq \tC_K(E)$ be the subspace
  of all $(\vec z_1,\dotsc,\vec z_r)$ such that each $\vec z_i$ is
  $\pi$-vertical. Then the action of $\frS_K$ on $\tC_K(E)$ restricts to
  $\tC^\pi_K(E)$ and we define $C^\pi_K(E)$ as the quotient.  We call these
  spaces \emph{ordered} and \emph{unordered vertical configuration spaces},
  respectively.

  The spaces $\tC^\pi_K(E)$ assemble into a functor
  $(\frS\wr\Inj)^{\text{op}}\longrightarrow \mathbf{Top}$ by permuting and
  omitting clusters as before. For a based symmetric sequence
  $\bmX=(X_k)_{k\ge 1}$, we define
  $C^\pi(E;\bmX)\coloneqq \int^K\tC_K^\pi(E)\times \bmX^K$. In other words,
  $C^\pi(E;\bmX)\subseteq C(E;\bmX)$ is the subspace of labelled configurations
  where each cluster is $\pi$-vertical.
\end{defi}

\begin{expl}
  For each $0\le p\le d$, we consider the projection $\pi\colon \R^d\to \R^p$ to
  the first $p$ coordinates, and define—for a tuple $K$ or a sequence $\bmX$,
  respectively—the spaces
  \begin{align*}
    C_K(\R^{p\kern-.6px,\kern.6pxd-p})&\coloneqq C_K^\pi(\R^d),\\
    C(\R^{p\kern-.6px,\kern.6pxd-p};\bmX)&\coloneqq C^\pi(\R^{d};\bmX).
  \end{align*}
  These are exactly the spaces depicted in \tref{Figure}{fig:manyV} from the
  introduction. Note that the subspace
  $C(\R^{p\kern-.6px,\kern.6pxd-p};\bmX)\subseteq C(\R^d;\bmX)$ is even an $E_d$-subalgebra: this
  follows directly from the observation that for each little cube
  $c\colon [0;1]^d\hookrightarrow [0;1]^d$ and a vertical cluster $\vec z$, the
  rescaled cluster $c(\vec z)$ is again vertical.
\end{expl}

Restricting the action of $\scC_d$ to its first $p$
coordinates, we can ask for the $p$-fold bar construction of
$C(\R^{p\kern-.6px,\kern.6pxd-p};\bmX)$, which still is an $E_{d-p}$-algebra. In order to formulate
our result, we need two more definitions:

\begin{defi}\label{def:eqWellBased}
  We call a based symmetric sequence $\bmX=(X_k)_{k\ge 1}$ \emph{equivariantly
    well-based} if each $*\hookrightarrow X_k$ is a
  cofibration in the projective model structure on $\frS_k$-spaces.
\end{defi}

\begin{defi}
  For a based symmetric sequence $\bm{X}=(X_k)_{k\ge 1}$, we define
  $\Sigma \bm{X}$ to be the sequence with $(\Sigma \bm{X})_k = \Sigma X_k$,
  together with the induced $\frS_k$-actions.
\end{defi}

\begin{theo}\label{thm:segal}
  If $\bmX$ is equivariantly well-based, then there is an equivalence of
  $E_{d-p}$-algebras
  \[B^pC(\R^{p\kern-.6px,\kern.6pxd-p};\bm{X})\simeq C(\R^{d-p};\Sigma^p \bm{X}).\]
\end{theo}

This equivalence is again a generalisation of Segal’s result \cite{Segal-1973}:
for each well-based space $X$, the labelled configuration space
$C(\R^{p\kern-.6px,\kern.6pxd-p};X[1])$ is isomorphic to $C(\R^d;X)$, since all
clusters have only a single particle, and hence \tref{Theorem}{thm:segal} boils
down to the well-known equivalence $B^pC(\R^d;X)\simeq C(\R^{d-p};\Sigma^p X)$
of $E_{d-p}$-algebras. In the case of $\bmX=X[k]$ for some well-based space $X$
with trivial $\frS_k$-action, we recover Theorem~\ref{thm:b}.

The proof of \tref{Theorem}{thm:segal} is nothing but a straightforward
generalisation of Segal’s proof, using at all stages that inside
$C(\R^{p\kern-.6px,\kern.6pxd-p};\bmX)$, clusters look as \emph{single}
particles from the perspective of the first $p$ coordinates.

\begin{proof}
  We strongly encourage the reader to compare the following proof to Segal’s
  original one \cite{Segal-1973}, as we shortened many arguments that can be
  copied verbatim. Throughout the proof, let us abbreviate $q\coloneqq d-p$.

  \medskip

  1.\hspace*{.5em}We first translate our statement into the language of
  \cite{Segal-1973} by considering a rectification of
  $C_{p\kern-.6px,\kern.6pxq}(\bmX)\coloneqq
  C(\R^{p\kern-.6px,\kern.6pxq};\bmX)$, which is even a true monoid: let
  $\pr_1\colon \R^d\to \R$ be the projection to the first coordinate; then we
  define the \emph{support} of
  \mbox{$c=\sum_i\vec z_i\otimes x_i\in C_{p\kern-.6px,\kern.6pxq}(\bmX)$} as
  $\supp(c)\coloneqq \bigcup_i\pr_1(\vec z_i)\subseteq\R$ and let
  \[\textstyle C'_{p\kern-.6px,\kern.6pxq}(\bmX)\coloneqq \mathopen{}\Big\{(t,c)\in \R_{\ge 0}\times
    C_{p\kern-.6px,\kern.6pxq}(\bmX);\,\supp(c)\subseteq (0;t)\Big\}\mathclose{}.\]
  By putting $(t,c)\cdot (t',c')=(t+t',c+T_tc')$, the space
  $C'_{p\kern-.6px,\kern.6pxq}(\bmX)$ becomes a topological monoid: here $T_t$
  is translation by $(t,0,\dotsc,0)$.  Note that
  $C'_{p\kern-.6px,\kern.6pxq}(\bmX)$ is the ‘Moore’ rectification
  $RC_{p\kern-.6px,\kern.6pxq}(\bmX)$ that appears in \cite[Prop.\,1.9]{Dunn}:
  its bar construction $BC'_{p\kern-.6px,\kern.6pxq}(\bmX)$ is, as an
  $E_{d-1}$-algebra, equivalent to the bar construction
  $BC_{p\kern-.6px,\kern.6pxq}(\bmX)$.  On the other hand, it follows from
  \cite[Cor.\,7.9]{Fiedorowicz} that $BC'_{p\kern-.6px,\kern.6pxq}(\bmX)$ can be
  calculated by the usual nerve construction for topological monoids (rather
  than the two-sided bar construction), which is a clustered version of the
  simplicial space that Segal studied. We show the analogue of
  \cite[Prop.\,2.1]{Segal-1973}: \emph{for each $p\ge 1$, we have an equivalence
    $BC'_{p\kern-.6px,\kern.6pxq}(\bmX)\simeq
    C_{p-1\kern-.3px,\kern.3pxq}(\Sigma\bmX)$ of $E_{d-1}$-algebras.}  Then the
  statement follows by induction.

  \medskip

  2.\hspace*{.5em}We consider the \emph{partial} abelian monoid
  $D_{p-1\kern-.3px,\kern.3pxq}(\bmX)$, whose underlying space\footnote{In
    contrast to Segal, we decided to introduce a new letter $D$ for this to
    avoid confusion when speaking of its bar construction.} is
  $C_{p-1\kern-.3px,\kern.3pxq}(\bmX)$, but where—instead of the
  $E_1$-multiplication—we call two labelled configurations summable if they are
  disjoint; in that case, the sum is their union. Recall that the classifying
  space of a partial monoid $M$ is the geometric realisation of its nerve
  $N_\bullet M$, where $N_nM\subseteq M^n$ contains composable
  $n$-tuples. Exactly as in \cite[Prop.\,2.3]{Segal-1973}, we obtain an
  isomorphism of $E_{d-1}$-algebras
  $BD_{p-1\kern-.3px,\kern.3pxq}(\bmX)\cong
  C_{p-1\kern-.3px,\kern.3pxq}(\Sigma\bmX)$ by amalgama\-ting the levelwise maps
  $\phi_n\colon N_nD_{p-1\kern-.3px,\kern.3pxq}(\bmX)\times\Delta^n\to
  C_{p-1\kern-.3px,\kern.3pxq}(\Sigma\bmX)$ with (writing
  $\Sigma X_k=X_k\wedge S^1$)\looseness-1\vspace*{-3px}
  \[\phi_n\mathopen\bigg(\hspace*{1px}\sum_{i=1}^{r_1}\vec z_{1,i}\otimes x_{1,i},
    \dotsc,\sum_{i=1}^{r_n}\vec z_{n,i}\otimes x_{n,i};t_1\le\dotsb\le t_n\hspace*{-1px}\bigg)\mathclose{}
    = \sum_{j=1}^n\sum_{i=1\vphantom{k}}^{\vphantom{J_g}r_{\kern-.4px\smash j}}\vec z_{j\kern-.3px,\kern.3pxi}\otimes (x_{j\kern-.3px,\kern.3pxi}\wedge t_j).\]

  \medskip

  3.\hspace*{.5em} We have a second projection $\pr_2\colon \R^{d}\to \R^{d-1}$
  and we call $\sum_i\vec z_i\otimes x_i\in C_{p\kern-.6px,\kern.6pxq}(\bmX)$ with $x_i\ne *$
  \emph{projectable} if the restriction $\pr_2|_{\smash{\bigcup_i[\vec z_i]}}$
  is injective. Let
  $C''_{p\kern-.6px,\kern.6pxq}(\bmX)\subseteq
  C'_{p\kern-.6px,\kern.6pxq}(\bmX)$ be the subspace of pairs $(t,c)$ with
  projectable $c$. Then $C''_{p\kern-.6px,\kern.6pxq}(\bmX)$ is a \emph{partial}
  submonoid with respect to the concatenation, where projectable configurations
  can be multiplied if their product is again projectable. Moreover, we have a
  map
  $C''_{\smash{p\kern-.6px,\kern.6pxq}}(\bmX)\to
  D_{\smash{p-1\kern-.3px,\kern.3pxq}}(\bmX)$ of partial monoids by projecting,
  see \tref{Figure}{fig:proj}.  As in \cite{Segal-1973}, the induced maps
  \mbox{$N_\bullet C''_{p\kern-.6px,\kern.6pxq}(\bmX)\to N_\bullet
    D_{p-1\kern-.3px,\kern.3pxq}(\bmX)$} between the spaces of composable tuples
  are homotopy equivalences and, since $\bm{X}$ was assumed to be equivariantly
  well-based, our simplicial spaces are proper, so we have a homotopy
  equivalence among the classifying spaces
  $BC''_{p\kern-.6px,\kern.6pxq}(\bmX)\to
  BD_{p-1\kern-.3px,\kern.3pxq}(\bmX)$.\looseness-1

  \medskip
  
  4.\hspace*{.5em}In the last step, which is a bit lengthy and which we
  outsource into \tref{Lemma}{lem:segalFinal}, we show that the inclusion
  $C''_{p\kern-.6px,\kern.6pxq}(\bmX)\subseteq
  C'_{p\kern-.6px,\kern.6pxq}(\bmX)$ of (partial) monoids induces a homotopy
  equivalence among classifying spaces: this is the analogue of
  \cite[Prop.\,2.4]{Segal-1973}. We therefore end up with a zig-zag of homotopy equivalences
  \[\begin{tikzcd}BC'_{p\kern-.6px,\kern.6pxq}(\bmX) &
      BC''_{p\kern-.6px,\kern.6pxq}(\bmX)\ar{r}[swap]{\simeq}\ar{l}{\simeq} &
      BD_{p-1\kern-.3px,\kern.3pxq}(\bmX)\ar{r}[swap]{\cong} &
      C_{p-1\kern-.3px,\kern.3pxq}(\Sigma\bmX).\end{tikzcd}\]%
  Since all three maps leave the remaining $d-1$ coordinates unchanged, they are
  morphisms of $E_{d-1}$-algebras, so $BC'_{p,q}(\bmX)$ and
  $C_{p-1,q}(\Sigma\bmX)$ are equivalent as $E_{d-1}$-algebras.
\end{proof}

\enlargethispage{\baselineskip}

We are left to show \tref{Lemma}{lem:segalFinal}. Even though the proof is both
technical and very similar to Segal’s one, we decided to spell out some details,
as they show at which stages the verticality constraint is used.

\begin{figure}
  \centering
  \begin{tikzpicture}[xscale=3.8,yscale=3]
  \draw[dgrey] (0,0) rectangle (1,1);
  \draw[dgrey,thick] (.7,.7) -- (.7,.9);
  \draw[dgrey,thick] (.5,.85) -- (.5,.35);
  \draw[dgrey,thick] (.3,.8) -- (.26,.8) -- (.26,.4) -- (.3,.4);
  \draw[dgrey,thick] (.3,.6) -- (.34,.6) -- (.34,.2) -- (.3,.2);
  \draw[dgrey,thick] (.3,.3) -- (.34,.3);
  \node at (.3,.2) {\tiny $\bullet$};
  \node at (.3,.3) {\tiny $\bullet$};
  \node at (.3,.4) {\tiny $\bullet$};
  \node at (.3,.6) {\tiny $\bullet$};
  \node at (.3,.8) {\tiny $\bullet$};
  \node at (.7,.7) {\tiny $\bullet$};
  \node at (.7,.9) {\tiny $\bullet$};
  \node at (.5,.85) {\tiny $\bullet$};
  \node at (.5,.35) {\tiny $\bullet$};
  \draw[dgrey] (1.5,0) -- (1.5,1);
  \draw[dgrey,thick] (1.5,.8) -- (1.46,.8) -- (1.46,.4) -- (1.5,.4);
  \draw[dgrey,thick] (1.5,.6) -- (1.54,.6) -- (1.54,.2) -- (1.5,.2);
  \draw[dgrey,thick] (1.5,.3) -- (1.54,.3);
  \draw[dgrey,thick] (1.5,.35) -- (1.42,.35) -- (1.42,.85) -- (1.5,.85);
  \draw[dgrey,thick] (1.5,.7) -- (1.54,.7) -- (1.54,.9) -- (1.5,.9);
  \node at (1.5,.2) {\tiny $\bullet$};
  \node at (1.5,.3) {\tiny $\bullet$};
  \node at (1.5,.4) {\tiny $\bullet$};
  \node at (1.5,.6) {\tiny $\bullet$};
  \node at (1.5,.8) {\tiny $\bullet$};
  \node at (1.5,.7) {\tiny $\bullet$};
  \node at (1.5,.9) {\tiny $\bullet$};
  \node at (1.5,.85) {\tiny $\bullet$};
  \node at (1.5,.35) {\tiny $\bullet$};
  \node at (1.215,.5) {\small $\mapsto$};
\end{tikzpicture}
  \caption{An instance of the map $C''_{p\kern-.6px,\kern.6pxq}(\bmX)
    \to D_{p-1\kern-.3px,\kern.3pxq}(\bmX)$}\label{fig:proj}
\end{figure}

\begin{lem}\label{lem:segalFinal}
  The inclusion $C''_{p\kern-.6px,\kern.6pxq}(\bmX)\subseteq C'_{p\kern-.6px,\kern.6pxq}(\bmX)$
  induces an equivalence on classifying spaces.
\end{lem}

\begin{proof}
  There is an equivalent description of $BM$ for a (partial) monoid $M$:
  consider the topological category $\textbf{C}(M)$ with object space $M$, and
  arrows $m\to m'$ being pairs $(m_1,m_2)\in M\times M$ with
  $m_1\cdot m\cdot m_2=m'$. Then $BM \cong |\bfC(M)|$, see
  \cite[Prop.\,2.5]{Segal-1973}.
  
  Let $Q$ be the space of triples $(a,b,c)$ with $a\le 0\le b$ and
  $c=\sum_i \vec z_i\otimes x_i\in C_{p\kern-.6px,\kern.6pxq}(\bmX)$ with
  support in $(a;b)$. We give $Q$ a partial order as follows: For each interval
  $L\subseteq \R$ and $c\in C_{p\kern-.6px,\kern.6pxq}(\bmX)$ whose support
  avoids $\partial L$, we define $c|_L$ as the subconfiguration that comprises
  all $\vec z_i\otimes x_i$ satisfying $\pr_1(\vec z_i)\in L$.  Here we use that
  $\pr_1(\vec z_i)$ is a single value in $\R$ by the verticality condition. Now
  we let $(a,b,c)\le (a',b',c')$ if $[a;b]\subseteq [a';b']$ and
  $c=c'|_{[a;b]}$, see \tref{Figure}{fig:ordered}. We get a functor
  $\pi\colon Q\to \bfC(C'_{p\kern-.6px,\kern.6pxq}(\bmX))$,
  \mbox{$\pi(a,b,c)\coloneqq (b-a,T_{-a}c)$} and we can copy
  \mbox{\cite[Lem.\,2.6]{Segal-1973}} verbatim to show that $\abs{\pi}$ is
  shrinkable, i.e.\ it has a section $s$ such that $s\circ|\pi|\simeq \on{id}$ by a homotopy $h_\bullet$ with $|\pi|\circ h_t=|\pi|$ for all $t$. Let
  $P\subseteq Q$ be the subspace of all $(a,b,c)$ with projectable $c$. Then
  $\pi(P)=C''_{p\kern-.6px,\kern.6pxq}(\bmX)$, so it is enough to show that
  $|P|\to |Q|$ is a homotopy equivalence. To do so, we use
  \cite[Prop.\,2.7]{Segal-1973}:

  \begin{figure}
  \centering
  \begin{tikzpicture}[xscale=3,yscale=2.7]
  \draw[thin,dgrey]  (.5,-.03) -- (.5,1.03);
  \draw[thin,dgrey] (0,0) rectangle (1,1);
  \node at (.5,-.08) {\tiny $0$};
  \draw[dgrey,thick] (.7,.7) -- (.7,.9);
  \draw[dgrey,thick] (.3,.8) -- (.26,.8) -- (.26,.4) -- (.3,.4);
  \draw[dgrey,thick] (.3,.6) -- (.34,.6) -- (.34,.2) -- (.3,.2);
  \draw[dgrey,thick] (.3,.3) -- (.34,.3);
  \node at (.3,.2) {\tiny $\bullet$};
  \node at (.3,.3) {\tiny $\bullet$};
  \node at (.3,.4) {\tiny $\bullet$};
  \node at (.3,.6) {\tiny $\bullet$};
  \node at (.3,.8) {\tiny $\bullet$};
  \node at (.7,.7) {\tiny $\bullet$};
  \node at (.7,.9) {\tiny $\bullet$};
  \node at (1.25,.5) {\small $\le$};
  \draw[thin,dgrey]  (2,-.03) -- (2,1.03);
  \draw[thin,dgrey] (1.5,0) rectangle (3,1);
  \node at (2,-.08) {\tiny $0$};
  \draw[dgrey,thick] (2.2,.7) -- (2.2,.9);
  \draw[dgrey,thick] (1.8,.8) -- (1.76,.8) -- (1.76,.4) -- (1.8,.4);
  \draw[dgrey,thick] (1.8,.6) -- (1.84,.6) -- (1.84,.2) -- (1.8,.2);
  \draw[dgrey,thick] (1.8,.3) -- (1.84,.3);
  \draw[dgrey,thick] (2.7,.4) -- (2.7,.9);
  \draw[dgrey,thick] (2.85,.7) -- (2.81,.7) -- (2.81,.4) -- (2.85,.4);
  \draw[dgrey,thick] (2.85,.6) -- (2.89,.6) -- (2.89,.25) -- (2.85,.25);
  \node at (1.8,.2) {\tiny $\bullet$};
  \node at (1.8,.3) {\tiny $\bullet$};
  \node at (1.8,.4) {\tiny $\bullet$};
  \node at (1.8,.6) {\tiny $\bullet$};
  \node at (1.8,.8) {\tiny $\bullet$};
  \node at (2.2,.7) {\tiny $\bullet$};
  \node at (2.2,.9) {\tiny $\bullet$};
  \node at (2.7,.4) {\tiny $\bullet$};
  \node at (2.7,.9) {\tiny $\bullet$};
  \node at (2.85,.25) {\tiny $\bullet$};
  \node at (2.85,.4) {\tiny $\bullet$};
  \node at (2.85,.6) {\tiny $\bullet$};
  \node at (2.85,.7) {\tiny $\bullet$};
\end{tikzpicture}
  \caption{The left configuration is \emph{smaller} than the right
    one since it is a restriction of the latter.}\label{fig:ordered}
\end{figure}

  \begin{quote}
    \textsc{Proposition.}~\itshape Let $Q$ be a good\footnote{
      A \emph{good ordered space} is an ordered space $Q$ such
      that its nerve is a good simplicial space. A topological monoid $(M,1)$ is good
      if $1$ has a contractible neighbourhood, see \cite[\textsc{a}2]{Segal-1973}.}
    ordered space such that:
    \begin{enumerate}[joinedup,packed]
    \item[\textsc{q1.}] For $\nu_1,\nu_2,\nu\in Q$ with $\nu_1,\nu_2\le \nu$
      there exists $\inf(\nu_1,\nu_2)$,
    \item[\textsc{q2.}] Wherever defined, $(\nu_1,\nu_2)\mapsto \inf(\nu_1,\nu_2)$
      is continuous,
    \end{enumerate}
    Moreover let $Q'\subseteq Q$ be open such that:
    \begin{enumerate}[joinedup,packed]
      \setcounter{enumi}{2}
    \item[\textsc{q3.}] For $\nu'\in Q'$ and $\nu\le \nu'$,
      we have $\nu\in Q'$,
    \item[\textsc{q4.}] There is a numerable open cover $(W_i)_{i\in I}$ and 
      $w_i\colon W_i\to Q'$ with $w_i(\nu)\le \nu$.
    \end{enumerate}
    Then $\abs{Q'}\to \abs{Q}$ is a homotopy equivalence.
  \end{quote}
  As in \cite[\textsc{a2}]{Segal-1973}, our special $Q$ is good; and additionally,
  the assumptions \textsc{q1} and \textsc{q2} are satisfied by the explicit
  construction of our order.
  
  Since $\bm{X}$ is equivariantly well-based, there are contractible and
  $\frS_k$-invariant neighbourhoods $U_k\subseteq X_k$ around the respective
  basepoints $*_k$, and equivariant homotopies moving $U_k$ into $*_k$.  We
  ‘thicken’ $P$ to an open subset $Q'\subseteq Q$ containing all
  configu\-rations that are projectable once we ignore clusters labelled in some
  $U_k$; we call these configurations \emph{almost projectable}. Then
  $\imath\colon |P|\to |Q'|$ is a homotopy equivalence, with retraction $\rho$
  given by forgetting clusters labelled in some $U_k$, the homotopy
  $\imath\circ\rho\simeq \mathrm{id}_{\smash{|Q'|}}$ induced by the homotopies
  from above. Moreover, $Q'$ satisfies \textsc{q3} since restrictions of
  projectables are still projectable. As a cover, we define, for each
  $n\ge 1$,\looseness-1
  \[W_n\coloneqq\mathopen{}\Big\{(a,b,c);\,c|_{\smash{[-\frac1n;\frac1n]}}
      \text{ is almost projectable}\Big\}\mathclose{}.\]
  Then $W_n\subseteq Q$ is open,
  $(W_n)_{n\ge 1}$ is numerable, and since each $c$ has only
  finitely many clusters, each $c$ admits a $n>0$ such that
  $c|_{\smash{[-\frac1n;\frac1n]}}$ projects to at most one
  point in $(-\frac1n;\frac1n)$; hence the restriction has to be
  projectable: therefore, $(W_n)_{n\ge 1}$ is exhaustive. Finally, the maps
  $w_n\colon W_n\to Q'$ with
  $w_n(a,b,c)\coloneqq (\max(a,-\frac1n),\min(b,\frac1n),c|_{\smash{[-\frac1n;\frac1n]}})$
  satisfy \textsc{q4}.
\end{proof}

Combining \tref{Theorem}{thm:BCRX} and \tref{Theorem}{thm:segal},
we obtain the following result:

\begin{cor}\label{cor:both}
  Let $\bmX$ be equivariantly well-based. Then we have an equivalence
  \[\textstyle B^{p+1}C(\R^{p\kern-.3px,\kern.2px1};\bmX)\simeq
    \Sigma^{p+1}\bigvee_{e\in \bbE}\Sigma^{p\cdot (\# e-1)}\bmX^{\wedge K(e)}.\]
\end{cor}

\section{Stable homology of vertical configuration spaces}
\label{sec:calc}
We want to use the previously established homotopical results for an explicit
homological statement about vertical configuration spaces: throughout this section,
let $p\ge 1$ and $k\ge 1$. By inserting a new $k$-cluster on the far right side, we
have stabilising maps
$C_{\smash{r\times k}}(\R^{p,1})\to C_{\smash{(r+1)\times k}}(\R^{p,1})$, see
\tref{Figure}{fig:stab}. Extending work of \cite{Tran,Palmer-2021,Latifi},
it is shown in \cite[Thm.\,4.3]{Bianchi-Kranhold} that the induced map in
$H_\bullet(-;\Z)$ is split injective, and an isomorphism if
$\bullet\le \frac{r}{2}$. We give a description of the stable
homology $H_\bullet(C_{\smash{\infty\times k}}(\R^{p,1}))$.

\begin{figure}[h]
  \centering
  \begin{tikzpicture}[xscale=3,yscale=2.6]
  \draw[thin,dgrey] (0,0) rectangle (1,1);
  \draw[dgrey,thick] (.7,.7) -- (.7,.5);
  \draw[dgrey,thick] (.7,.3) -- (.7,.2);
  \draw[dgrey,thick] (.4,.87) -- (.4,.35);
  \draw[dgrey,thick] (.3,.8) -- (.26,.8) -- (.26,.4) -- (.3,.4);
  \draw[dgrey,thick] (.3,.6) -- (.34,.6) -- (.34,.3) -- (.3,.3);
  \node at (.3,.3) {\tiny $\bullet$};
  \node at (.3,.4) {\tiny $\bullet$};
  \node at (.3,.6) {\tiny $\bullet$};
  \node at (.3,.8) {\tiny $\bullet$};
  \node at (.7,.5) {\tiny $\bullet$};
  \node at (.7,.7) {\tiny $\bullet$};
  \node at (.7,.2) {\tiny $\bullet$};
  \node at (.7,.3) {\tiny $\bullet$};
  \node at (.4,.87) {\tiny $\bullet$};
  \node at (.4,.35) {\tiny $\bullet$};
\end{tikzpicture}\quad\raisebox{34px}{$\mapsto$}\quad
\begin{tikzpicture}[xscale=3,yscale=2.6]
  \draw[thin,bgrey] (1,0) -- (1,1);
  \draw[thin,dgrey] (0,0) rectangle (1.5,1);
  \draw[dgrey,thick] (.7,.7) -- (.7,.5);
  \draw[dgrey,thick] (.7,.3) -- (.7,.2);
  \draw[dgrey,thick] (.4,.87) -- (.4,.35);
  \draw[dgrey,thick] (.3,.8) -- (.26,.8) -- (.26,.4) -- (.3,.4);
  \draw[dgrey,thick] (.3,.6) -- (.34,.6) -- (.34,.3) -- (.3,.3);
  \draw[bblue,thick] (1.25,.4) -- (1.25,.6);
  \node at (.3,.3) {\tiny $\bullet$};
  \node at (.3,.4) {\tiny $\bullet$};
  \node at (.3,.6) {\tiny $\bullet$};
  \node at (.3,.8) {\tiny $\bullet$};
  \node at (.7,.5) {\tiny $\bullet$};
  \node at (.7,.7) {\tiny $\bullet$};
  \node at (.7,.2) {\tiny $\bullet$};
  \node at (.7,.3) {\tiny $\bullet$};
  \node at (.4,.87) {\tiny $\bullet$};
  \node at (.4,.35) {\tiny $\bullet$};
  \node[dblue] at (1.25,.4) {\tiny $\bullet$};
  \node[dblue] at (1.25,.6) {\tiny $\bullet$};
\end{tikzpicture}
  \caption{The stabilisation $C_{5\times 2}(\R^{1,1})\to C_{6\times 2}(\R^{1,1})$}\label{fig:stab}
\end{figure}

\enlargethispage{\baselineskip}

\begin{constr}[Coloured configuration spaces]\label{constr:colConf}
  Let $I$ be an index set and $\alpha=(\alpha_i)_{i\in I}$ be a finitely
  supported family of non-negative integers (i.e.\ $\alpha_i\ne 0$ for only
  finitely many $i\in I$). For each space $E$, the group
  $\prod_{i\in I}\frS_{\alpha_i}$ acts freely on $\tC_{\smash{|\alpha|}}(E)$,
  and we define the \emph{coloured configuration space}
  \[\textstyle C^\alpha(E) \coloneqq \tC_{|\alpha|}(E)/\prod_{i\in I}\frS_{\alpha_i}.\]
  This definition is rather similar to the one of the clustered configuration
  space $C_\alpha(E)$ from \tref{Definition}{def:clustConf}, but we quotient out
  a bit less: intuitively, a point in $C^\alpha(E)$ is
  a disjoint configuration of unordered \emph{coloured} particles, exactly
  $\alpha_i$ particles of colour $i$.
  Coloured configuration spaces have been studied in \cite{Palmer-2018}.

  A \emph{parity map} is an assignment $t\colon I\to \Z_2$: it merely divides
  $I$ into ‘odd’ and ‘even’ colours. For each parity map and each finitely
  supported tuple $\alpha=(\alpha_i)_{i\in I}$, we have a sign function
  $\prod_i \frS_{\alpha_i}\to \{\pm 1\}$ sending $(\sigma_i)_{i\in I}$ to
  the product of signs
  $\prod_i \on{sg}(\sigma_i)^{\smash{t(i)}}$. Via the canonical projection
  $\pi_1(C^\alpha(E))\to \prod_i \frS_{\alpha_i}$, this gives rise to a local
  system $\epsilon^\alpha$ on $C^\alpha(E)$. If the parity map is
  clear from the context, we write\looseness-1
  \[M_\bullet(E;\alpha) \coloneqq H_\bullet(C^\alpha(E);\epsilon^\alpha).\]
\end{constr}

Although \tref{Construction}{constr:colConf} might seem unrelated at first
glance, the modules $M_\bullet(E;\alpha)$ are useful to describe the homology
of vertical configuration spaces:

\begin{defi}\label{defi:TRS}
  Let $\bbE[k]\subseteq\bE$ be the subset of entanglement types
  $e=(S_1,\dotsc,S_w)$ such that each $S_i$ is of cardinality $k$.  Fixing a
  dimension $p\ge 1$, the parity of $e$ is defined to be $p\cdot (\# e-1)$, that
  is: an entanglement type is \emph{even} if it has odd weight or if $p$ is
  even.

  For each finite tuple $\alpha=(\alpha_e)_{\smash{e\in\bbE[k]}}$, we define
  $r(\alpha)\coloneqq \sum_e \alpha_e\cdot \# e$ and
  $s(\alpha)\coloneqq r(\alpha)-|\alpha|$. Intuitively, $\alpha$ tells us how
  often which entanglement type can be seen, $r(\alpha)$ tells us how many
  \emph{clusters} are involved, and $s(\alpha)$ measures the difference between
  the number of clusters and the number of entanglement types.
\end{defi}

In \cite[§\,4]{Bianchi-Kranhold}, we introduce a filtration
$\scF_\bullet C_{\smash{r\times k}}(\R^{p,1})$, and
in \cite[Prop.\,4.19]{Bianchi-Kranhold}, we establish an isomorphism
of graded abelian groups
\begin{align}\label{eq:BK}
  H_\bullet(\scF_sC_{\smash{r\times k}}(\R^{p,1}),
  \scF_{\smash{s-1}}C_{\smash{r\times k}}(\R^{p,1}))\cong
  \bigoplus_{(r(\alpha),s(\alpha))=(r,s)}
  M_{\bullet-p\cdot s}(\R^{p+1};\alpha)
\end{align}
as follows: given a coloured configuration, we ‘insert’, at each particle in
$\R^{p+1}$ of colour $e$, a standard configuration that realises the
entanglement type $e$ along a vertical line. The degree shift and the sign
system is caused—via the Thom isomorphism—by small perturbations of the
clusters, tracking all possibilities how to ‘break’ an entanglement.\looseness-1

However, we could not determine if the Leray spectral sequence
associated with the above filtration collapses on its first page and if
the extension problem is trivial \cite[Outl.\,4.22]{Bianchi-Kranhold}: this
would imply that
\mbox{$H_\bullet(C_{\smash{r\times k}}(\R^{p,1}))\cong \bigoplus_{r(\alpha)=r} M_{\bullet-p\cdot
    s(\alpha)}(\R^{p+1};\alpha)$}.
We show that this is at least stably the case.

\begin{constr}[Stabilisation]
  Let $I$ be an index set as before, and we pick a distinguished colour
  $i_0\in I$.  If $\lambda=(\lambda_i)_{i\in I\setminus\{i_0\}}$ is a finitely
  supported tuple of integers $\lambda_i\ge 0$ and $n\ge |\lambda|$ is an
  integer, then we let $\lambda[n]$ be the $I$-indexed tuple that additionally
  contains the entry $\lambda_{\smash{i_0}}=n-|\lambda|$.

  Adding a point of colour $i_0$ on the far right side gives rise to a
  stabilisation map $C^{\lambda[n]}(\R^d)\to C^{\lambda[n+1]}(\R^d)$ among
  coloured configuration spaces. For each parity map, the local system
  $\epsilon^{\smash{\lambda[n+1]}}$ restricts to $\epsilon^{\smash{\lambda[n]}}$ along
  the stabilisation map, and for the case in which $i_0$ has even parity, it is shown in
  \cite[Lem\,4.21]{Bianchi-Kranhold} that the induced map
  \mbox{$M_\bullet(\R^d;\lambda[n])\to M_\bullet(\R^d;\lambda[n+1])$} is split
  injective, and bijective for $\smash{\bullet\le \frac{n-|\lambda|}{2}}$: this
  is a signed version of \cite[Cor.\,\textsc{c}]{Palmer-2018}. We define the
  stable module
  \[M_\bullet(\R^d;\lambda[\infty])\coloneqq \varinjlim_n M_\bullet(\R^d;\lambda[n]).\]
\end{constr}

\begin{expl}
  There is a single entanglement type $e_0=(\set{1,\dotsc,k})\in\bbE[k]$ of
  weight $1$; it clearly has even parity for each $p\ge 1$.  Adding a cluster
  $C_{\smash{r\times k}}(\R^{p,1})\to C_{\smash{(r+1)\times k}}(\R^{p,1})$
  preserves the aforementioned filtration from \cite{Bianchi-Kranhold}, and
  translates via (\ref{eq:BK}) to the stabilisations
  $C^{\lambda[n]}(\R^{p+1})\to C^{\lambda[n+1]}(\R^{p+1})$ by adding a particle
  of colour $e_0$. This was the key ingredient for the proof of homological
  stability \cite[Thm.\,4.3]{Bianchi-Kranhold}. Finally, note that 
  $s(\lambda[n])$ from \tref{Definition}{defi:TRS} is independent of $n$, so we
  can just write $s(\lambda)$.
\end{expl}

\setcounter{aThm}{2}
\begin{aThm}\label{thm:stableH}
  For each $p\ge 1$, we have an isomorphism of graded abelian groups
  \[\textstyle H_\bullet(C_{\smash{\infty\times k}}(\R^{p,1}))\cong
    \bigoplus_{\lambda} M_{\bullet-p\cdot s(\lambda)}(\R^{p+1};\lambda[\infty]),\]
  where $\lambda$ ranges in the set of finitely supported tuples indexed by $\bbE[k]\setminus\{e_0\}$.
\end{aThm}

Before proving the theorem, we note that for $k=1$, both sides are clearly the
same: as the verticality condition becomes empty, we have
$C_{\smash{\infty\times k}}(\R^{p,1})=C_\infty(\R^{p+1})$.  On the other hand, since
$\bbE[1]$ contains \emph{only} the distinguished entanglement type $e_0$, the
only possible $\lambda$ is the empty tuple, and in this case, $s(\lambda)=0$ and
$M_\bullet(\R^{p+1};\lambda[\infty])$ is just the stable homology of the sequence of spaces
$C_n(\R^{p+1})$.

\begin{proof}
  As before, let $\S^0[k]$ be the based symmetric sequence whose
  $k$\textsuperscript{th} space is the $0$-sphere $\S^0$, and whose
  remaining constituents are trivial.
  As in \tref{Remark}{rem:labEffect}, the labelled vertical configuration
  space $C(\R^{p,1};\S^0[k])$ is isomorphic to $\coprod_r C_{\smash{r\times k}}(\R^{p,1})$.
  Since $p\ge 1$, $C(\R^{p,1};\S^0[k])$ is at least an $E_2$-algebra,
  in particular $H$-commutative. Hence the
  group completion theorem \cite[Prop.\,1]{McDuff-Segal} applies and we calculate the
  stable homology as
  \[H_\bullet\mathopen{}\big(C_{\infty\times k}(\R^{p,1})\big)\mathclose{}
    \cong H_\bullet\mathopen{}\big(\Omega_{0}^{\smash{p+1}\vphantom{G_a}}
    B^{\smash{p+1}\vphantom{G_a}}_{\vphantom{0}}C(\R^{p,1};\S^0[k])\big)\mathclose{},\]
  where $\Omega_0$ denotes the path component of the constant loop.  Using
  \tref{Corollary}{cor:both}, we obtain
  $B^{p+1}C(\R^{p,1};\S^0[k])\simeq \Sigma^{p+1}\bigvee_e \S^{p\cdot (\# e-1)}$,
  where $e$ ranges in $\bbE[k]$. Now we use that this space can be desuspended
  $p+1$ times, i.e.\ we calculate the stable homology of a \emph{free}
  $E_{p+1}$-algebra: the bouquet $\bigvee_e \S^{p\cdot (\# e-1)}$ has
  two path components, namely $\bigvee_{\# e\ge 2}\bS^{p\cdot (\# e-1)}$,
  which also contains the basepoint, and $\{e_0\}$. 
  If we let $C_m\subseteq C(\R^{p+1};\bigvee_e \bS^{p\cdot (\# e-1)})$ be the
  component of configurations with exactly $m$ particles labelled by $e_0$, then we have stabilisations
  $C_m\to C_{m+1}$ by adding a particle with label $e_0$, and we denote its colimit by $C_\infty$.
  By applying the group completion theorem once again, we obtain
  \begin{align*}
    \textstyle H_\bullet\pa{C_{\infty\times k}(\R^{p,1})}
    \cong \textstyle H_\bullet\mathopen{}\big(\Omega^{\smash{p+1}\vphantom{G_a}}_0\Sigma^{\smash{p+1}\vphantom{G_a}}_{\vphantom 0}
      \bigvee_{e}\bS^{p\cdot (\# e-1)}\big)\mathclose{}
    \cong \textstyle H_\bullet\mathopen{}\big(C_\infty\big)\mathclose{}.
  \end{align*}
  The space $C_m$ admits a stable splitting
  $\smash{\Sigma^\infty_+ C_m\simeq \Sigma^\infty \bigvee_\alpha D^\alpha}$ in
  the spirit of \cite{Snaith-1974}, where $\alpha$ ranges in
  tuples with $\alpha_{e_0}=m$, and $D^\alpha$ is the subspace
  of configurations that have, for each $e$, at most
  $\alpha_e$ particles with labels in the sphere corresponding to $e$, quotiented
  by the subspace of configurations where at least one of these labels is the basepoint.\looseness-1
  
  As in \cite[§\,2.6]{CFB-Cohen-Taylor}, $D^\alpha$ is the Thom space of a disc
  bundle over $C^\alpha(\R^{p+1})$ (whose sign system is exactly $\epsilon^\alpha$),
  so we get a Thom isomorphism
  $\smash{\tilde H_\bullet(D^\alpha) \cong M_{\bullet-p\cdot
    s(\alpha)}(\R^{p+1};\alpha)}$. Altogether, we have
  $H_\bullet(C_m)\cong \bigoplus_\alpha M_{\bullet-p\cdot
    s(\alpha)}(\R^{p+1};\alpha)$, where $\alpha$ ranges in tuples
  with $\alpha_{e_0}=m$.  Under this identification, the stabilisation maps
  $H_\bullet(C_m)\to H_\bullet(C_{\smash{m+1}})$ split as the sum of stabilising maps
  $M_{\bullet-p\cdot s(\lambda)}(\R^{p+1};\lambda[n])\to M_{\bullet-p\cdot
    s(\lambda)}(\R^{p+1};\lambda[n+1])$, indexed by all $\lambda$ and with
  $n=m+|\lambda|$. This proves the claim.
\end{proof}

\printbibliography[heading=bibintoc]

\small

\noindent\textbf{Florian Kranhold}\quad
Karlsruher Institut für Technologie,
  Fakultät für Mathematik,
  Institut für Algebra und Geometrie,
  Englerstraße 2,
  76131 Karlsruhe,
  Germany,
\mail{kranhold@kit.edu}

\end{document}